\documentclass[12pt,usenames,dvipsnames]{amsart}

\usepackage{mathpazo}
\usepackage{hyperref}
\usepackage{cleveref}
\usepackage{bm}
\usepackage{comment}
\usepackage[margin=1in]{geometry}
\usepackage{enumitem}
\usepackage{calrsfs}
\DeclareMathAlphabet{\pazocal}{OMS}{zplm}{m}{n}

\usepackage{graphicx}
\usepackage{tikz-cd}
\usepackage{multirow}

\usepackage[cmtip, all]{xy}
\usepackage{array}
\usepackage{lscape}

\usepackage{amssymb,amsmath,latexsym,amsthm}

\newcommand{\bS}{\mathbb{S}}

\DeclareMathOperator{\Hom}{Hom}

\DeclareMathOperator{\SL}{SL}

\DeclareMathOperator{\PGL}{PGL}
\DeclareMathOperator{\GL}{GL}

\DeclareMathOperator{\Span}{span}

\DeclareMathOperator{\id}{id}

\DeclareMathOperator{\Dr}{Dr}

\DeclareMathOperator{\init}{in}

\DeclareMathOperator{\Star}{Star}
\DeclareMathOperator{\Gr}{Gr}
\DeclareMathOperator{\TGr}{TGr}

\DeclareMathOperator{\sgn}{sgn}

\DeclareMathOperator{\Conv}{conv}

\DeclareMathOperator{\Fl}{F\ell}
\DeclareMathOperator{\TFl}{TF\ell}

\DeclareMathOperator{\Trop}{Trop}

\DeclareMathOperator{\lc}{lc}

\DeclareMathOperator{\op}{op}

\newcommand{\pB}{\pazocal{B}}
\newcommand{\pC}{\pazocal{C}}

\newcommand{\field}[1]{\mathbb{#1}}

\newcommand{\Z}{\field{Z}}

\newcommand{\R}{\field{R}}
\newcommand{\C}{\field{C}}

\renewcommand{\P}{\field{P}}

\newcommand{\GG}{\mathbb{G}}
\newcommand{\bk}[2]{\langle #1, #2 \rangle}

\newcommand{\Zone}{\langle \mathbf{1} \rangle}
\newcommand{\bone}{\mathbf{1}}
\newcommand{\pF}{\pazocal{F}}

\newcommand{\pQ}{\pazocal{Q}}
\newcommand{\boldS}{\mathbf{S}}
\newcommand{\vecomega}{\vec{\omega}}

\newcommand{\Berg}{\mathsf{Berg}}
\newcommand{\MS}{\mathsf{MS}}
\newcommand{\Quiv}{\mathsf{Quiv}}
\newcommand{\Simp}{\mathsf{Simp}}
\newcommand{\lp}{\mathsf{lp}}

\newcommand{\sa}{\mathsf{a}}
\newcommand{\se}{\mathsf{e}}
\newcommand{\su}{\mathsf{u}}
\newcommand{\sv}{\mathsf{v}}
\newcommand{\sw}{\mathsf{w}}

\newcommand{\sA}{\mathsf{A}}

\newcommand{\sP}{\mathsf{P}}
\newcommand{\sQ}{\mathsf{Q}}
\newcommand{\sR}{\mathsf{R}}

\newcommand{\Sn}[1]{\mathfrak{S}_{#1}}
\newcommand{\onto}{\twoheadrightarrow}
\newcommand{\chow}[2]{#1/\!\!\!/#2}

\newcommand{\Flo}[1]{\Fl^{\circ}(#1)}
\newcommand{\TFlo}[1]{\TFl^{\circ}(#1)}

\DeclareMathOperator{\semigp}{semigp}

\newtheorem{theorem}{Theorem}[section]
\newtheorem{lemma}[theorem]{Lemma}
\newtheorem{proposition}[theorem]{Proposition}

\newtheorem*{theorem*}{Theorem}

\theoremstyle{definition}

\newtheorem{example}[theorem]{Example}
\newtheorem{remark}[theorem]{Remark}

\theoremstyle{remark}

\numberwithin{equation}{section}
\numberwithin{table}{section}
\numberwithin{figure}{section}

\title{Initial degenerations of flag varieties}

\author{Daniel Corey}
\author{Jorge Alberto Olarte}

\begin{document}
	
	\maketitle
	
	\begin{abstract}
	We prove that the initial degenerations of the flag variety admit closed immersions into finite inverse limits of flag matroid strata, where the diagrams are derived from matroidal subdivisions of a suitable flag matroid polytope. As an application, we prove that the initial degenerations of $\Fl^{\circ}(n)$---the open subvariety of the complete flag variety $\Fl(n)$ consisting of flags in general position---are smooth and irreducible when $n\leq 4$. We also study the Chow quotient of $\Fl(n)$ by the diagonal torus of $\PGL(n)$, and show that, for $n=4$, this is a log crepant resolution of its log canonical model.
    
    \medskip
    
    
    \medskip
    
\end{abstract}

\section{Introduction}
	
	Flag varieties live at the intersection of algebraic geometry, representation theory and combinatorics, and their degenerations have been intensely studied from several perspectives. A standard goal is to find flat degenerations of the flag variety to toric varieties, as this  allows one to compute numerical invariants of the flag varieties using combinatorial techniques. The theories of Newton-Okounkov bodies \cite{HaradaKaveh, Kaveh, KoganMiller} and cluster algebras \cite{BossingerFriasMedinaMageeNajeraChavez,GrossHackingKeelKontsevich, RietschWilliams} has been successful in this endeavor\footnote{See \cite{FangFourierLittelmann} for a comprehensive history of this subject.}, building on earlier work in \cite{AlexeevBrion, Caldero}.  In contrast, the goal of this paper is to systematically investigate all initial degenerations---those arising via Gr\"obner theory and tropical geometry---of the Lie-type $\mathsf{A}$ flag varieties, and use this to study their Chow quotients. 
	
	The Lie-type $\sA$ flag variety $\Fl(\vec{r},n)$, where $n\geq 1$ and  $\vec{r}=(r_1,r_2,\ldots,r_s)$ is an increasing sequence of positive integers less than $n$, parameterizes \textit{flags}, i.e., sequences of linear subspaces $F_1\subset F_2 \subset \cdots \subset F_s$ of $\C^{n}$ where $\dim F_{k} = r_k$.
	When $s=1$, the flag variety specializes to the \textit{Grassmannian} $\Gr(r,n)$. 
	The Pl\"ucker embeddings of the Grassmannians $\Gr(r_i,n)$ combine to produce an embedding
	\begin{equation*}
	    \Fl(\vec{r},n) \hookrightarrow \P(\wedge^{r_1}\C^{n}) \times \cdots \times \P(\wedge^{r_s}\C^{n}),
	\end{equation*}
    which gives $\Fl(\vec{r},n)$ the structure of a multiprojective variety. The multi-homogeneous coordinates of a flag are called its  \textit{Pl\"ucker coordinates}. Let $\Fl^{\circ}(\vec{r},n)$ be the open locus of $\Fl(\vec{r},n)$ whose Pl\"ucker coordinates are all nonzero. Its tropicalization $\TFl^{\circ}(\vec{r},n)$ may be viewed through two frameworks. Via Gr\"obner theory, a vector $\vec{\sw}\in \TFl^{\circ}(\vec{r},n)$ induces an initial degeneration $\init_{\vec{\sw}}\Fl^{\circ}(\vec{r},n)$ of the flag variety. Such a vector $\vec{\sw}$ also carries the data of a valuated flag matroid, which induces a coherent mixed subdivision $\pQ(\vecomega)$ of the Minkowski sum of hypersimplices
    \begin{equation*}
        \Delta(\vec{r},n) := \Delta(r_1,n) + \cdots + \Delta(r_s,n) 
    \end{equation*}
    into \textit{flag matroid polytopes} \cite{BrandtEurZhang}.    Each $\C$-realizable flag matroid defines a locally-closed subscheme $\Fl(\vec{\sQ})$ of  $\Fl(\vec{r},n)$, and the assignment $\vec{\sQ} \mapsto \Fl(\vec{\sQ})$ defines a diagram of type $\pQ(\vec{\omega})$, and we may form its inverse limit $\varprojlim_{\pQ(\vec{\omega})} \Fl$. 
    
    \begin{theorem}
    \label{thm:closedimmersionintro}
	There is a closed immersion
	\begin{equation}\label{eq:closedImmersionIntro}
	    \init_{\vec{\sw}}\Fl^{\circ}(\vec{r},n) \hookrightarrow \varprojlim_{\pQ(\vec{\omega})} \Fl.
	\end{equation}
	\end{theorem}

	This generalizes \cite[Theorem~1.1]{CoreyGrassmannians}, where the first author proves that each initial degeneration of the Grassmannian embeds into an inverse limits of thin Schubert cells. This is used to prove that, for $n=7$ or $8$, the Chow quotient of $\Gr(3,n)$ by the diagonal torus of $\PGL(n)$ is the log canonical compactification of the moduli space of $n$ points in $\P^{2}$ in linear general position (up to projective transformations), resolving a conjecture of Hacking, Keel, and Tevelev \cite[Conjecture~1.6]{KeelTevelev2006} (the case $n=7$ appears in \cite{CoreyGrassmannians} and $n=8$ in \cite{CoreyLuber}, which uses in a critical way \cite[Proposition~7.9]{Schock}). A Lie-type $\mathsf{D}$ generalization is studied in \cite{CoreySpinor}, where it is shown that each initial degeneration of the spinor variety embeds into an inverse limit of even $\Delta$-matroid strata.
	
	Nevertheless, there are some notable differences between the (type-$\mathsf{A}$) flag variety in this paper and the Grassmannian and spinor varieties. To a parabolic subgroup $P$ of a semisimple algebraic group $G$ is associated a weight $\mu$, a unique irreducible representation $V_{\mu}$ of highest weight $\mu$, and a canonical embedding $G/P \hookrightarrow \P(V_{\mu})$. This is the Pl\"ucker embedding when $G/P = \Gr(r,n)$ and the Wick embedding when $G/P = \mathbb{S}_n$.  
	In these two cases, the decomposition of $G/P$ induced by the toric stratification of $\P(V_{\mu})$ are the matroid and $\Delta$-matroid decompositions of $\Gr(r,n)$ and $\bS_{n}$, respectively, and this fact is crucial to the main results of \cite{CoreySpinor,CoreyGrassmannians}.  The analogous statement is not true for $\Fl(\vec{r},n)$, in particular, the intersection of $\Fl(\vec{r},n)$ with the dense torus of $\P(V_{\mu})$ is not $\Fl^{\circ}(\vec{r},n)$, see Remark \ref{rmk:SchurNoWork}. However, the toric stratification of $\P(\wedge^{r_1}\C^{n}) \times \cdots \times \P(\wedge^{r_s}\C^{n})$  \textit{does} induce the decomposition of $\Fl(\vec{r},n)$ into flag matroid strata.

	To illustrate the utility of Theorem \ref{thm:closedimmersionintro}, we study the initial degenerations of the complete flag variety $\Fl(n):= \Fl((1,2,\ldots,n-1),n)$. In this case, the polytope $\Delta(\vec{r},n)$ is the $(n-1)$--dimensional permutahedron $\Pi_{n}$, the flag matroid strata $\Fl(\vec{\sQ})\subset \Fl(n)$ admit simple parameterizations,  and, for $n\leq 4$, the closed immersion \eqref{eq:closedImmersionIntro} is an isomorphism for each $\vec{\sw}\in \TFl^{\circ}(n)$.  With a careful analysis of the diagrams of flag matroid strata parameterized by the matroidal subdivisions of the permutahedron $\Pi_4$, we prove the following theorem.  
	
	\begin{theorem}
	\label{thm:Flag4intro}
	The initial degenerations of $\Fl^{\circ}(4)$ are smooth and irreducible.
	\end{theorem}
	
	We conclude by studying the action of the diagonal torus $H\subset \PGL(n)$ on the complete flag variety $\Fl(n)$ and the Chow quotient $\chow{\Fl(n)}{H}$. Recall that $H$ acts on $\Gr(r,n)$ in the following way. If $F$ is the row span of the full-rank $(r\times n)$--matrix $A$, and $h\in H$, then $F\cdot h$ is the row-span of the matrix $Ah$. This action preserves inclusion of subspaces, and hence induces an action on $\Fl(\vec{r},n)$:
	\begin{equation}
	\label{eq:HonFl}
	    (F_1\subset \cdots \subset F_s) \cdot h = (F_1\cdot h \subset \cdots \subset F_s\cdot h).
	\end{equation}
	This restricts to a free action on the open locus $\Fl^{\circ}(\vec{r},n)$. 
	In the case of the complete flag variety, the quotient $\Fl^{\circ}(n)/H$ has a modular interpretation, inspired by \cite[p.183]{Hu}.   The group $H$ acts freely and transitively on the set of general 1-dimensional in linear subspaces $F_1\subset \C^n$. So in the orbit $(F_1\subset F_2 \cdots \subset F_{n-1})\cdot H$, there is a unique representative of the form  $(\Zone \subset \widetilde{F}_2 \cdots \subset \widetilde{F}_{n-1})$, where $\bone = (1,\ldots,1)$. Therefore, the quotient $\Fl^{\circ}(n)/H$ may be identified with:
	\begin{equation*}
	    \Fl_{\bone}^{\circ}(n-1) := \{ (F_{2} \subset F_3 \subset \cdots \subset F_{n-1})  \, : \, (\Zone \subset F_{2} \subset F_3 \subset \cdots \subset F_{n-1}) \in \Fl^{\circ}(n)\}.
	\end{equation*}
	
	\begin{theorem}
	\label{thm:chowIntro}
	The Chow quotient $\chow{\Fl(4)}{H}$ is smooth and a log crepant resolution of the log canonical compactification of $\Fl_{\bone}^{\circ}(3)$. 
	\end{theorem}
	
	Here is an outline of the paper.  We begin in \S \ref{sec:matroidalSubdivisionsOfGeneralizedPermutahedra} by reviewing coherent mixed subdivisions of Minkowski sums of point configurations, then discuss flag matroids and their flag Dressians. In \S\ref{sec:flagMatroids}, we describe flag matroid strata scheme-theoretically and prove Theorem \ref{thm:closedimmersionintro}. In \S\ref{sec:flagFaces} we review the constructions of flag matroids that allow us to describe the faces of flag matroid polytopes. We specialize to the complete flag variety in \S\ref{sec:completeFlag}; there we describe how to parameterize complete flag matroid strata. We use this description to prove that the flag matroid strata of $\Fl(4)$ are smooth and irreducible, and that the morphisms between such strata are smooth and dominant with connected fibers. Theorem \ref{thm:Flag4intro} is proved in \S\ref{sec:inverseLimitsCompleteFlag}. In the last section \S\ref{sec:Chow}, we discuss generalities on the Chow quotient $\chow{\Fl(\vec{r},n)}{H}$, then prove Theorem \ref{thm:chowIntro}. There are two appendices.   We summarize the notation in Appendix \ref{app:notation}. In Appendix \ref{app:data}, we record the flag matroid strata in $\Fl(4)$, and the flag matroidal subdivisions of the $3$-dimensional permutahedron. 
	
	\subsection*{Acknowledgments} We thank Michael Joswig and Dhruv Ranganathan for helpful conversations and comments on a previous draft.  
	
	\subsection*{Funding} DC is supported by the SFB-TRR project ``Symbolic Tools in
Mathematics and their Application'' (project-ID 286237555). JAO is supported by  Deutsche Forschungsgemeinschaft (DFG, German Research Foundation) under Germany's Excellence Strategy - The Berlin Mathematics Research Center MATH$^+$ (EXC-2046/1, project ID 390685689). 
	
	\subsection*{Code} We use \texttt{polymake.jl}  \cite{polymake,polymakeJL} and OSCAR \cite{OSCAR-book,OSCAR} to generate the data in Appendix \ref{app:data}. These are used in the proofs of Propositions \ref{prop:flagStrata4}, \ref{prop:flagStrataMaps4}, and Theorem \ref{thm:initFl4}; we also use this software in the proof of Lemma \ref{lem:notPreservedUnderTranslation}.
	The code can be found at the following github repository:
	
	\begin{center}
	    \url{https://github.com/dcorey2814/initialDegenerationsFlagVarieties}
	\end{center}

    \section{Matroidal subdivisions of generalized permutahedra}
	\label{sec:matroidalSubdivisionsOfGeneralizedPermutahedra}

	We begin by recalling Minkowski sums of point configurations and their coherent mixed subdivisions in a general framework. We then discuss their connection to flag matroids via the flag Dressian of \cite{BrandtEurZhang}. 
	
	\subsection{Weighted point configurations}
	\label{sec:weightedPointConfigurations}
	For  $n\geq 1$, let \label{notn:setn} $[n] = \{1,\ldots,n\}$ and \label{notn:allone} $\bone=(1,\ldots,1)$ in any $\Z^{n}$.  Let \label{notn:N} $N = \Z^{n}/\Zone$, \label{notn:M}  $M=\Hom(N,\Z)$, and \label{notn:uv} $\langle \su, \sv \rangle$ the canonical pairing of  $(\su,\sv) \in M\times N$.   A (integral) \label{notn:pointConf} \textit{point configuration} is a pair $(\sQ, \sa)$ where \label{notn:Q} $\sQ$ is a finite set and $\sa:\sQ \to M$ is a function; informally, this is a finite collection of  $\sQ$-labeled points in $M$ where points that coincide are distinguished by their labels. The polytope associated to $(\sQ,\sa)$ is 
	\begin{equation*}
	\label{notn:DeltaQ}
	    \Delta(\sQ,\sa) = \Conv\{ \sa(\lambda) \, : \, \lambda \in \sQ  \}
	\end{equation*}
	When $\sa$ is clear from the context, we drop it from the notation.	We refer the reader to  \cite{DeLoeraRambauSantos} for a comprehensive treatment of point configurations, their faces, and polyhedral subdivisions. Notably, many terms normally associated to polytopes (like face and dimension) are instead applied to the labeling set $\sQ$. 
	
	Let \label{notn:NQ} $N(\sQ) = \Z^{\sQ} / \Zone$. A \textit{weighted point configuration} is a triple \label{notn:omega} $\omega = (\sQ,\sa,\sw)$ where $(\sQ,\sa)$ is a point configuration and $\sw\in N(\sQ)$. Given $\sv\in N_{\R}$, let
	\begin{equation*}
	\label{notn:Qvw}
	\sQ_{\sv}^{\sw} = \{\lambda \in \sQ \, : \, \bk{\sa(\lambda)}{\sv} + \sw_{\lambda} \leq \bk{\sa(\xi)}{\sv} + \sw_{\xi} \text{ for all } y \in \sQ\}.
	\end{equation*}
    When $\sw=0$, we write \label{notn:Qv} $\sQ_{\sv} = \sQ_{\sv}^{0}$; any subset of this form is called a \textit{face} of $(\sQ,\sa)$. The \textit{coherent subdivision associated to} $\omega$ is
	\begin{equation*}
	\label{notn:subdOmega}
	\pQ(\omega) = \{\sQ_{\sv}^{\sw} \, : \, \sv \in N_{\R}\}. 
	\end{equation*} 
	Geometrically, the subdivision $\pQ(\omega)$ is obtained by lifting the points  $\sa(\lambda)$ in $M \times \R$ to heights dictated by $\sw$, and projecting the lower faces back down to $M$.

 	Consider point configurations $(\sQ_i,\sa_i)$ for $i=1\ldots,s$; set  \label{notn:vecQ} $\vec{\sQ} = (\sQ_1,\ldots,\sQ_s)$ and $\vec{\sa} = (\sa_1,\ldots,\sa_s)$.  The \textit{Minkowski sum} of $(\sQ_1,\sa_1),\ldots,(\sQ_s,\sa_s)$, denoted \label{notn:MinkowskiSum} $\MS(\vec{\sQ},\vec{\sa})$, is the point configuration $(\sQ_1\times \ldots \times \sQ_s, \sa)$ where 
 	 \begin{equation*}
 	\sa:\sQ_1 \times \cdots \times \sQ_s \to M \hspace{20pt}
 (\lambda_1, \ldots, \lambda_s) \mapsto \sa_1(\lambda_1) + \cdots + \sa_s(\lambda_s).
	 \end{equation*}
	 Under the natural identification $N(\sQ_1 \times \cdots \times \sQ_s) = (\Z^{\sQ_1} \otimes \cdots \otimes \Z^{\sQ_s}) / \Zone$, define
	 \begin{align}
	 \label{eq:mapphi}
	 \begin{array}{rl}
	 \phi: N(\sQ_1) \times \cdots \times N(\sQ_s) &\to N(\sQ_1 \times \cdots \times \sQ_s) \\
	      (\sw_1,\ldots,\sw_s) &\mapsto (\sw_1 \otimes \bone \otimes \cdots \otimes \bone) + \cdots + (\bone  \otimes \cdots \otimes \bone \otimes \sw_s)
	 \end{array}
	 \end{align}	
 	Consider weighted point configurations  $\omega_i = (\sQ_i, \sa_i, \sw_i)$ for $i=1\ldots,s$ and set \label{notn:vecomega} $\vecomega = (\omega_1,\ldots,\omega_s)$. The \textit{Minkowski sum} of $\vecomega$, denoted $\MS(\vec{\omega})$, is the weighted point configuration $(\sQ,\sa, \sw)$ where $(\sQ,\sa) = \MS(\vec{\sQ},\vec{\sa})$ and 
	 $\sw=\phi(\sw_1,\ldots,\sw_s)$.  The \textit{coherent mixed subdivision associated to} $\vecomega$ is  
	 \begin{equation*}
	 \label{notn:subdVecOmega} 
	 \pQ(\vecomega) := \pQ(\MS(\vecomega)). 
	 \end{equation*}
	  The faces of $\pQ(\vecomega)$ can be described in the following way.  Given $\sv$ in $N_{\R}$ and $\vec{\sw} = (\sw_1,\ldots,\sw_s)$ in $N(\sQ_1)\times \cdots \times N(\sQ_s)$, then, with $\MS(\vec{\sQ},\vec{\sa}) = (\sQ,\sa)$, we have 
	  \begin{equation}
	    \label{eq:facesMS}
    	(\sQ_{\sv}^{\phi(\vec{\sw})}, \sa) = \MS(\vec{\sQ}_{\sv}^{\vec{\sw}}, \vec{\sa}) \hspace{20pt} \text{where} \hspace{20pt} \vec{\sQ}_{\sv}^{\vec{\sw}} = ((\sQ_{1})_{\sv}^{\sw_1}, \ldots, (\sQ_{s})_{\sv}^{\sw_s}).
	  \end{equation}
	
	\begin{remark}
	Strictly speaking, $\vecomega$ denotes a sequence of weighted point configurations $(\omega_1,\ldots,\omega_s)$. Nevertheless, it is useful to use the shorthand $\vecomega = (\vec{\sQ}, \vec{\sa}, \vec{\sw})$ where $\vec{\sQ} = (\sQ_1,\ldots,\sQ_s)$, $\vec{\sa} = (\sa_1, \ldots, \sa_s)$, and $\vec{\sw} = (\sw_1, \ldots, \sw_s)$.
	\end{remark}
	
	\begin{remark}
	    The definitions made here are consistent with the standard ones as described in \cite{DeLoeraRambauSantos}. For instance, the coherent subdivision $\pQ(\omega)$ agrees with Definition 2.2.10 of \textit{loc. cit.}, and the coherent mixed subdivision $\pQ(\vecomega)$ agrees with Definition 9.2.7 of  \textit{loc. cit.}, see also \cite[\S~2.2]{HuberRambauSantos}. 
	\end{remark}
	
		\begin{figure}[tbh!]
        \centering
        \includegraphics[height=5cm]{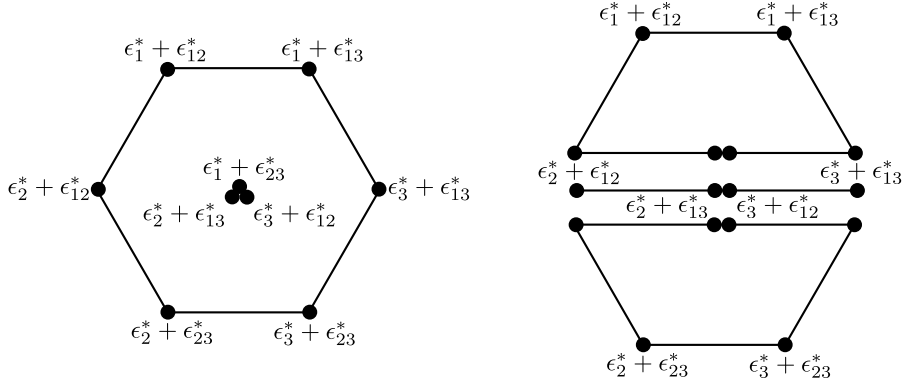}
        \caption{A coherent mixed subdivision}
         \label{fig:hexagon}
    \end{figure}
    
	\begin{example}
	\label{ex:hexagon}
        Consider the point configurations $(\sQ_1, \sa_1)$ and $(\sQ_2,\sa_2)$ where $\sQ_1 = \{1,2,3\}$, $\sQ_{2} = \{12, 13, 23\}$, and
        \begin{equation*}
         \sa_{i}: \sQ_{i} \to M \hspace{20pt} \lambda \mapsto \epsilon_{\lambda}^{*}  
        \end{equation*}
        Set $\vec{\sQ} = (\sQ_1, \sQ_2)$ and $\vec{\sa} = (\sa_1, \sa_2)$. 
        The Minkowski sum $\MS(\vec{\sQ}, \vec{\sa})$ is illustrated in Figure \ref{fig:hexagon}(left).  Now let $\vec{\sw} = \se_1$ and $\vecomega = (\vec{\sQ}, \vec{\sa}, \vec{\sw})$. Then 
        \begin{equation}
            \phi(\vec{\sw}) = \se_1 \otimes \se_{12} + \se_1 \otimes \se_{13} + \se_1 \otimes \se_{23}.
        \end{equation}
        The function $\phi(\vec{\sw})$ lifts the points $\epsilon_1^{*} + \epsilon_{12}^{*}$, $\epsilon_1^{*} + \epsilon_{13}^{*}$,  and $\epsilon_1^{*} + \epsilon_{23}^{*}$ to height 1, and  the coherent mixed subdivision $\pQ(\vecomega)$ is illustrated in Figure \ref{fig:hexagon} (right). Observe that $\epsilon_1^{*} + \epsilon_{23}^{*}$ does not appear in any subpoint configuration of the subdivision. 
	\end{example}

	The various subdivisions $\pQ(\vecomega)$ define a complete polyhedral fan $\pF(\vec{\sQ},\vec{\sa})$ in $N(\sQ_1) \times \cdots \times N(\sQ_s)$ which is a special instance of the \textit{fiber fan} (as we describe later in this section):  $\vec{\sw}$ and $\vec{\sw}'$ are contained in the relative interior of the same cone of $\pF(\vec{\sQ},\vec{\sa})$ if $\pQ(\vecomega) = \pQ(\vecomega')$, where $\vecomega = (\vec{\sQ}, \vec{\sa}, \vec{\sw})$ and $\vecomega' = (\vec{\sQ}, \vec{\sa}, \vec{\sw}')$. When $(\vec{\sQ},\vec{\sa})$ consists of a single point configuration $(\sQ,\sa)$, we simply write $\pF(\sQ,\sa)$ for this fan.

	The fan $\pF(\vec{\sQ}, \vec{\sa})$ is preserved under translation by a linear space that we now describe. Denote by \label{notn:se} $\se_{\lambda}$ for $\lambda \in \sQ$ the images of the standard bases vectors under $\Z^{\sQ} \to N(\sQ)$. 
	Define maps $b_i:N \to N(\sQ_i)$ by 
	\begin{equation*}
	    b_i(\sv) = \sum_{\lambda\in \sQ_i} \bk{\sa_i(\lambda)}{\sv} \; \se_{\lambda}. 
	\end{equation*}
	That is, $b_i(\sv)$ is the weight vector on $(\sQ_i,\sa_i)$ induced by the linear function $\sv$ on $M$, and hence the subdivision associate to $(\sQ_i,\sa_i,b_i(\sv))$ is trivial. 
	Let $L$ denote saturation of the image of the map
	\begin{equation}
	\label{eq:definingL}
	    N \to N(\sQ_1) \times \cdots \times N(\sQ_s) \hspace{20pt} \sv \mapsto (b_1(\sv),\ldots,b_s(\sv)).
	\end{equation}
	Then \label{notn:fiberFan} $\pF(\vec{\sQ},\vec{\sa})$ is preserved under translation by $L_{\R}$. If the polytope $\Delta(\MS(\vec{\sQ}, \vec{\sa}))$ affinely spans $M$, then $L_{\R}$ is the lineality space of $\pF(\vec{\sQ},\vec{\sa})$, and hence $\pF(\vec{\sQ},\vec{\sa})/ L_{\R}$ is a pointed polyhedral fan. This fan is closely related to the fiber polytope \cite{BilleraSturmfels}, as we now describe; this characterization is essential for our study of the Chow quotient $\Fl(\vec{r},n)$ in \S\ref{sec:Chow}. Let \label{notn:MQ} $M(\sQ) = \Hom(N(\sQ),\Z)$, and \label{notn:sedual} $\se_{\lambda}^{*}$ the dual of $\se_{\lambda}$. The simplex $\Simp(\sQ)$ in $M(\sQ)$ is given by
    \begin{equation*}
        \Simp(\sQ) = \Conv\{\se_{\lambda}^{*} \, : \, \lambda \in \sQ \}.
	\end{equation*}
	Given point configurations $(\sQ_1,\sa_1), \ldots, (\sQ_s,\sa_s)$, consider the projection
	\begin{equation}
	\label{eq:fiberPolytopeMap}
	    \pi:M(\sQ_1) \times \cdots \times M(\sQ_s) \to M 
	    \hspace{20pt} (\se_{\lambda_1}^{*},\ldots,\se_{\lambda_s}^{*}) \mapsto \sa_1(\lambda_1) + \cdots + \sa_s(\lambda_s),
	\end{equation}
	thus $\pi$ takes the vertices of $\Simp(\sQ_1) \times \cdots \times \Simp(\sQ_s)$ to the points in the configuration $\MS(\vec{\sQ},\vec{\sa})$. To the projection $\pi$ is associated the \textit{fiber polytope} whose faces are in bijection with the coherent mixed (i.e., $\pi$-induced) subdivisions of $\MS(\vec{\sQ},\vec{\sa})$, and its normal fan is naturally identified with $\pF(\vec{\sQ},\vec{\sa})/ L_{\R}$.

\subsection{Flag matroids}
    We assume that the reader is familiar with basic notions of matroid theory; see \cite{Oxley} for a general reference. There are many ways to characterize a matroid, we emphasize the definition via bases. Given $1\leq r \leq n$, denote by \label{notn:binomnr} $\binom{[n]}{r}$ the  $r$-element subset of $[n]$.  A $(r,n)$--\textit{matroid} is a nonempty subset  $\sQ$ of $\binom{[n]}{r}$ that satisfies the basis-exchange axiom. Of particular importance is the matroid of a $r$-dimensional linear subspace of $F\subset \C^n$.  Denote by $\epsilon_1,\ldots,\epsilon_n$ the standard basis of $\Z^n\subset \C^{n}$, and let $\pi_F:\C^{n} \to \C^{n}/F$ be the quotient map. Define $\sQ^*(F)$ to be the matroid of the vector configuration $\{\pi_F(\epsilon_i) \, : \, i=1,\ldots,n \}$, i.e., a subset $\lambda \subset [n]$ is a basis of $\sQ^*(F)$ if and only if the set of vectors $\{\pi_F(\epsilon_i) \, : \, i\in \lambda\}$ is a vector space basis of $\C^{n} / F$. The \textit{matroid of $F$}, written $\sQ(F)$, is the dual of  $\sQ^{*}(F)$, and a matroid of the form $\sQ(F)$ is \textit{$\C$-realizable}. See Formula \ref{eq:realizableMatroid} for an alternate characterization of $\sQ(F)$.

	Let $\sQ_1$ and $\sQ_2$ be matroids on $[n]$ of ranks $r_1\leq r_2$, respectively. The matroid $\sQ_1$ is a \textit{quotient} of $\sQ_2$, written $\sQ_2\onto \sQ_1$, if every flat of $\sQ_1$ is a flat of $\sQ_2$. If $F_1\subset F_2$ are linear subspaces of $\C^{n}$ of dimensions $r_1$ and $r_2$ respectively, then the quotient map $\C^{n}/F_1 \onto \C^{n}/F_2$ induces a matroid quotient $\sQ^{*}(F_1) \onto \sQ^{*}(F_2)$, which yields a matroid quotient $\sQ(F_2) \onto \sQ(F_1)$ by duality. 
	
	Given an increasing sequence of integers $\vec{r} = (r_1,\ldots,r_s)$ less than $n$, a $(\vec{r},n)$-\textit{flag matroid} is a sequence of matroids $\vec{\sQ} = (\sQ_1,\ldots,\sQ_s)$ on $[n]$ such that $\sQ_j$ has rank $r_j$ and $\sQ_i$ is a quotient of $\sQ_j$ for $i\leq j$. The \textit{bases} of $\vec{\sQ}$ is the set
	\begin{equation*}
    \label{notn:basesFlagMatroid}
	\pB(\vec{\sQ}) = \bigcup_{j=1}^{s} \sQ_j.
	\end{equation*}
	Given a flag $F_* = (F_1\subset \ldots \subset F_{s}\subset \C^{n})$ in $\Fl(\vec{r},n)$, its flag matroid is 
	\begin{equation}
	\label{eq:flagMatroidSubspaces}
	\vec{\sQ}(F_*) = (\sQ(F_1),\ldots,\sQ(F_s)).
	\end{equation}
	
	\subsection{The flag Dressian}
	Denote by \label{notn:epsilon} $\epsilon_1,\ldots,\epsilon_n$ the images of the standard basis vectors under the quotient $\Z^{n} \to N$ and \label{notn:epsilondual} $\epsilon_{i}^{*}\in M$ the dual of $\epsilon_i$. Given $\lambda = \{i_1,\ldots,i_r\} \in \binom{[n]}{r}$, we set \label{notn:epsilonlambda} $\epsilon_{\lambda} = \epsilon_{i_1} + \cdots + \epsilon_{i_r}$ and  $\epsilon_{\lambda}^{*} = \epsilon_{i_1}^* + \cdots + \epsilon_{i_r}^*$.

	Let $\sQ$ be a $(r,n)$--matroid. Its point configuration is $(\sQ,\sa)$ where $\sa(\lambda) = \epsilon_{\lambda}^{*}$; any point configuration of this form is said to be \textit{matroidal}. The polytope $\Delta(\sQ,\sa)$, which we abbreviate to $\Delta(\sQ)$ when $\sQ$ is a matroid, is the \textit{matroid polytope} of $\sQ$. 
	Let $\sw\in N(\sQ)_{\R}$ and $\omega = (\sQ,\sa,\sw)$. The coherent subdivision $\pQ(\omega)$ is \textit{matroidal} if, for any $\sv\in N_{\R}$, the subset $\sQ_{\sv}^{\sw} \subset \sQ$ is a matroid. Explicitly,  
	\begin{equation}
	\label{eq:Qvw}
	\sQ_{\sv}^{\sw} = \{ \lambda \in \sQ \, : \, \bk{\epsilon_{\lambda}^*}{\sv} + \sw_{\lambda} \leq \bk{\epsilon_{\xi}^{*}}{\sv} + \sw_{\xi} \text{ for all } \xi \in \sQ \}
	\end{equation} 
	The \textit{Dressian} of $\sQ$ \cite{HerrmannJensenJoswigSturmfels}, written $\Dr(\sQ)$, is 
	\begin{equation*}
	\Dr(\sQ) = \{\sw \in N(\sQ)_{\R} \, : \, \pQ(\omega) \text{ is matroidal, where } \omega=(\sQ,\sa,\sw)\}
	\end{equation*}
	which is the support of a subfan $\pF_{\Dr}(\sQ)$ of $\pF(\sQ,\sa)$. 
	
	Now let $\vec{\sQ} = (\sQ_1,\ldots,\sQ_s)$ be a flag matroid. Its point configuration is the Minkowski sum $\MS(\vec{\sQ},\vec{\sa})$. The polytope $\Delta(\MS(\vec{\sQ},\vec{\sa}))$ is called a \textit{flag matroid polytope}, and is denoted by $\Delta(\vec{\sQ})$. 
	Let $\vec{\sw} \in N(\sQ_1)_{\R} \times \dots \times N(\sQ_s)_{\R}$ and $\vecomega=(\vec{\sQ},\vec{\sa},\vec{\sw})$. The subdivision $\pQ(\vecomega)$ is \textit{matroidal} if for any $\sv\in N_{\R}$, the sequence $\vec{\sQ}_{\sv}^{\vec{\sw}} = ((\sQ_1)_{\sv}^{\sw_1}, \ldots, (\sQ_s)_{\sv}^{\sw_s})$ is a flag matroid. The sequence $\vec{\sQ} = (\sQ_1, \sQ_2)$ from Example \ref{ex:hexagon} is a flag matroid, and 	the subdivision $\pQ(\vecomega)$ is matroidal. 
	
	The \textit{flag Dressian} of $\vec{\sQ}$ is the set
	\begin{equation*}
	\label{notn:flagDr}
	\Dr(\vec{\sQ}) = \{\vec{\sw} \in \Dr(\sQ_1)\times \cdots \times \Dr(\sQ_s)  \, : \, \pQ(\vec{\omega}) \text{ is matroidal, where } \vecomega =(\vec{\sQ},\vec{\sa},\vec{\sw})  \}.
	\end{equation*}
	The flag Dressian $\Dr(\vec{\sQ})$ is the support of a subfan $\pF_{\Dr}(\vec{\sQ})$ of $\pF(\vec{\sQ},\vec{\sa})$.

    \begin{remark}
    \label{rmk:matroidPolytope}
    When $\vec{\sQ}=(\sQ_1,\ldots,\sQ_s)$ is the uniform $(\vec{r},n)$--flag matroid, i.e., where each $\sQ_{k}$ is the uniform $(r_k,n)$ matroid, the flag matroid polytope of $\vec{\sQ}$ is $\Delta(\vec{r},n)$ from the introduction.  When $\vec{r} = (1,\ldots,n-1)$, this polytope is the  $(n-1)$--dimensional permutahedron \label{notn:permutahedron} $\Pi_n$. Polytopes of the form $\Delta(\vec{\sQ})$ for $(\vec{r},n)$--flag matroids $\vec{\sQ}$ are exactly the generalized permutahedra whose vertices are among those of $\Delta(\vec{r},n)$.  
    \end{remark}
	
	\begin{remark}
	\label{rmk:valuatedMatroids}
		A vector $\sw \in \Dr(\sQ)$ is a \textit{valuated matroid} with underlying matroid $\sQ$, in the sense of \cite{DressWenzel1992}, see \cite[Proposition~2.2]{Speyer2008} and \cite[\S 4.4]{MaclaganSturmfels2015}.  The flag Dressian is defined in \cite{BrandtEurZhang}, and a tuple of vectors $\vec{\sw} \in \Dr(\vec{\sQ})$ is a \textit{valuated flag matroid}.
	\end{remark}

\section{Flag matroid strata}
    \label{sec:flagMatroids}
    
    The coherent mixed subdivisions $\pQ(\vecomega)$ from the previous section serve as the posets of the finite inverse limits appearing in the Introduction. In this section, we complete the definition of $\varprojlim_{\pQ(\vecomega)} \Fl$ by describing  flag matroid strata and the morphisms between them, and prove Theorem \ref{thm:closedimmersionintro}. 
    
    \subsection{Tropicalization}
In this subsection, we recall the process of tropicalizating over a trivially valued field from the Gr\"obner viewpoint, see \cite[Chapters~2-3]{MaclaganSturmfels2015} for a detailed introduction. 

Denote by \label{notn:TQ}$T(\sQ) = \GG_m^{\sQ} / \GG_m$ the torus with cocharacter lattice $N(\sQ) = \Z^{\sQ}/\Zone$. Let $I \subset \C[x_{\lambda} \, : \, \lambda \in \sQ]$ be a homogeneous prime ideal that does not contain a monomial. Its extension  $I^{\circ}$ to  $\C[x_{\lambda}^{\pm} \, : \, \lambda \in \sQ]$ defines an closed irreducible subvariety  $X^{\circ}$ of $T(\sQ)$. Given $\sw \in N(\sQ)_{\R}$ and $f\in \C[x_{\lambda}^{\pm} \, : \, \lambda \in \sQ]$, the \textit{initial form} of $f$ is
\begin{equation*}
\init_{\sw} f = \sum_{\substack{\su:\bk{\su}{\sw} \\ \text{ is minimal}}} c_{\su} x^{\su}
\hspace{20pt} \text{ where } \hspace{20pt} f = \sum c_{\su} x^{\su}.
\end{equation*}
The $\sw$-\textit{initial ideal} of $I$, respectively $I^{\circ}$, is 
\begin{equation*}
\init_{\sw} I = \langle \init_{\sw} f \, : \, f\in I   \rangle, \hspace{20pt}     \init_{\sw} I^{\circ} = \langle \init_{\sw} f \, : \, f\in I^{\circ}   \rangle. 
\end{equation*}
The \textit{tropicalization} of $X^{\circ}$ is
\begin{equation*}
\Trop(X^{\circ}) = \{ \sw \in N(\sQ)_{\R}  \, : \, \init_{\sw} I^{\circ} \neq \langle 1 \rangle  \}
\end{equation*}
Given $\sw \in \Trop(X^{\circ})$, the $\sw$-\textit{initial degeneration} of $X^{\circ}$ is the subscheme of $T(\sQ)$ cut out by $\init_{\sw}I^{\circ}$. 
The \textit{Gr\"obner fan} of $I$ is a complete rational polyhedral fan in $N_{\R}$ where $\sw$ and $\sw'$ belong to the relative interior of the same cone if $\init_{\sw}I = \init_{\sw'}I$. Thus $\Trop(X^{\circ})$ is the support of a subfan of the Gr\"obner fan of $I$ and there are finitely many distinct $\init_{\sw} X^{\circ}$.

	\subsection{Flag matroid strata}
    As a set, the Grassmannian $\Gr(r,n)$ consists of the $r$-dimensional vector subspaces of $\C^{n}$. It may be realized as a closed subvariety of $\P(\wedge^{r}\C^n)$ via the Pl\"ucker embedding
	\begin{equation*}
	\iota_{r,n}: \Gr(r,n) \hookrightarrow \P(\wedge^{r}\C^n) \hspace{30pt} F \mapsto [\zeta_{\lambda}(F) \, : \, \lambda \in \textstyle{\binom{[n]}{r}}].
	\end{equation*}
	The homogeneous coordinates $\zeta_{\lambda}(F)$ are called the \textit{Pl\"ucker coordinates} of $F$. Implicit in this description is the natural identification of $\wedge^{r}\C^{n}$ with $\C^{\binom{[n]}{r}}$. In terms of Pl\"ucker coordinates, the matroid $\sQ(F)$ of a subspace $F\subset \C^{n}$ defined in \S \ref{sec:matroidalSubdivisionsOfGeneralizedPermutahedra} is
	\begin{equation}
	\label{eq:realizableMatroid}
	\sQ(F) = \{\lambda \in \textstyle{\binom{[n]}{r}} \, : \, \zeta_{\lambda}(F) \neq 0 \};
	\end{equation}

    Given $\sQ\subset \binom{[n]}{r}$, we identify $T(\sQ) = \GG_{m}^{\sQ}/\GG_m$ with the dense torus of the coordinate stratum $\P(\C^{\sQ}) \subset \P(\wedge^r\C^n)$. The \textit{stratum} of a  $(r,n)$--matroid $\sQ$, denoted by $\Gr(\sQ)$, is the scheme-theoretic intersection
	\begin{equation*}
	\Gr(\sQ) = \Gr(r,n) \cap T(\sQ).
	\end{equation*}
	The $\C$-valued points of $\Gr(\sQ)$ correspond to $r$-dimensional subspaces of $F\subset \C^{n}$ such that $\sQ(F) = \sQ$. 
	
	Fix integers $1\leq r_1 < \cdots < r_s < n$ and set $\vec{r} = (r_1,\ldots,r_s)$. As a set, the \textit{flag variety}  $\Fl(\vec{r},n)$ is
	\begin{equation*}
	\Fl(\vec{r},n) = \{F_* = (F_1\subset \cdots \subset F_s \subset \C^{n}) \, : \, \dim F_i = r_i   \}
	\end{equation*}
	The embedding 
	\begin{equation*}
	\Fl(\vec{r},n) \hookrightarrow \P(\wedge^{r_1}\C^{n}) \times \cdots \times \P(\wedge^{r_s}\C^{n}) \hspace{30pt} F_{*} \mapsto (\iota_{r_1,n}(F_1), \ldots, \iota_{r_s,n}(F_s)).
	\end{equation*}
	realizes $\Fl(\vec{r},n)$ as a multiprojective variety; denote by
	\begin{equation*}
	I_{\vec{r},n} \subset \bigotimes_{j=1}^{s} \C[p_{\lambda} \, : \, \lambda \in \textstyle{\binom{[n]}{r_j}}]
	\end{equation*}
	its (multihomogeneous) ideal. It is generated by the polynomials
	\begin{equation*}
	f_{\mu,\nu} =  \sum_{k\in \nu \setminus \mu} (-1)^{\sgn(k;\mu,\nu)} p_{\mu\cup k} p_{\nu \setminus k}
	\end{equation*}
	where $|\nu\setminus \mu| \geq 3$, $|\mu| = r_{i}-1$, and $|\nu| = r_{j} + 1$; here $\sgn$ is the sign function defined in \cite[p.169]{MaclaganSturmfels2015}. 
	
	 Let $\vec{\sQ} = (\sQ_1,\ldots,\sQ_s)$ be a $(\vec{r},n)$--flag matroid. As before, denote by $T(\sQ_i)$ the dense torus of the coordinate stratum $\P(\C^{\sQ}) \subset \P(\wedge^{r_i}\C^{n})$, so $T(\sQ_1) \times \cdots \times T(\sQ_s) \cong (\GG_m^{\sQ_1} \times \cdots \times \GG_{m}^{\sQ_s} ) / \GG_m^s$. The \textit{stratum} of $\vec{\sQ}$, denoted by $\Fl(\vec{\sQ})$, is the scheme-theoretic intersection 
	\begin{equation*}
	\label{notn:FlvecQ}
	\Fl(\vec{\sQ}) = \Fl(\vec{r},n) \cap (T(\sQ_1) \times \cdots \times T(\sQ_s)).
	\end{equation*}
	and its $\C$-valued points correspond to the flags $F_{*}$ in $\Fl(\vec{r},n)$ such that $\vec{\sQ}(F_*) = \vec{\sQ}$, see Formula \eqref{eq:flagMatroidSubspaces}. 
	
	\begin{remark}
	\label{rmk:SchurNoWork}
	One may realize $\Fl(\vec{r},n)$ as a $G/P$ where $G=\SL(n)$ and $P$ is a standard parabolic subgroup.  Associated to $P$ is a weight $\mu$ and a partition $\lambda$. The irreducible representation $V_{\mu}$ of $G$ with highest weight $\mu$ is the Schur functor $\boldS_{\lambda}(\C^{n}) \subset \wedge^{r_1}\C^{n} \otimes \cdots \otimes \wedge^{r_s}\C^{n}$, and the canonical embedding $G/P \hookrightarrow \P(V_{\mu})$ is  $\Fl(\vec{r},n) \hookrightarrow \P(\boldS_{\lambda}\C^{n})$.  Concretely, $\Fl(\vec{r},n)$ embeds into $\P(\wedge^{r_1}\C^{n} \otimes \cdots \otimes \wedge^{r_s}\C^{n})$ by composing $\Fl(\vec{r},n)\hookrightarrow \P(\wedge^{r_1}\C^{n}) \times \cdots \times \P(\wedge^{r_s}\C^{n})$ with the Segre embedding $\P(\wedge^{r_1}\C^{n}) \times \cdots \times \P(\wedge^{r_s}\C^{n}) \hookrightarrow \P(\wedge^{r_1}\C^{n} \otimes \cdots \otimes \wedge^{r_s}\C^{n})$, and the linear span of its image is exactly $\P(\boldS_{\lambda}\C^{n})$ \cite[\S 9]{FultonYT}. Using this description, one readily verifies that, e.g., the intersection of $\Fl(\vec{r},n)$ with the dense torus of $\P(\boldS_{\lambda}\C^{n})$ is \textit{not} $\Fl^{\circ}(\vec{r},n)$ when $\vec{r}$ has at least $2$ components (even for $\Fl((1,2),3)$). Nevertheless, the decomposition of $\Fl(\vec{r},n)$ into flag matroid strata is induced by the toric stratification of the multi-projective space $\P(\wedge^{r_1}\C^{n}) \times \cdots \times \P(\wedge^{r_s}\C^{n})$.   
 	\end{remark}

	Let us describe the coordinate ring of $\Fl(\vec{\sQ})$. Define
	\begin{itemize}
		\item[-] $B_{\vec{\sQ}} = \C[p_{\lambda} \, : \, \lambda \in \sQ_1] \otimes \cdots \otimes \C[p_{\lambda} \, : \, \lambda \in \sQ_s]$; 
		\item[-] $I_{\vec{\sQ}} = (I_{\vec{r},n} + \langle p_{\lambda} \, : \, \lambda \notin \pB(\vec{\sQ}) \rangle) \cap B_{\vec{\sQ}}$;
		\item[-] $S_{\vec{\sQ}} = \langle p_{\lambda} \, :\, \lambda \in \pB(\vec{\sQ}) \rangle_{\semigp}$.
	\end{itemize}
	Then the coordinate ring of $\Fl(\vec{\sQ})$ is the $\GG_m^s$-invariants subring of $R_{\vec{\sQ}} := S_{\vec{\sQ}}^{-1}B_{\vec{\sQ}} / I_{\vec{\sQ}}$. The ideal $I_{\vec{\sQ}}$ is generated by the polynomials
	\begin{equation}
	\label{eq:flagMatroidIdealGens}
	f_{\mu,\nu}(\vec{\sQ}) =  \sum_{\substack{k\in \nu \setminus \mu \\ \mu \cup k, \nu \setminus k \in \pB(\vec{\sQ})}  \\  } (-1)^{\sgn(k;\mu,\nu)} p_{\mu\cup k} p_{\nu \setminus k}
	\end{equation}
	where $i\leq j$, $|\nu \setminus \mu| \geq 3$, $\mu$ is an independent set of $\sQ_{i}$ with $|\mu| = r_i-1$, and $\nu$ has rank $r_j$ in $\sQ_j$ with $|\nu|= r_j + 1$.
	
\begin{remark}
\label{rmk:repeatedRanks}
At times, especially in \S \ref{sec:flagFaces}, it is convenient to relax the strict inequalities $r_i<r_j$ to $r_i\leq r_j$ in the definition of $\Fl(\vec{r},n)$ and $(\vec{r},n)$--flag matroids. This does not add any generality. For example, if $r_{i}=r_{i+1}$, then the $s$-step flag $(F_1\subset \cdots \subset F_{i} \subset F_{i+1} \subset \cdots \subset F_s)$ in $\Fl(\vec{r},n)$ is canonically identified with the $(s-1)$--step flag $(F_1\subset \cdots \subset F_{i-1} \subset F_{i+1} \subset \cdots \subset F_s)$ as $F_{i}=F_{i+1}$. Therefore, $\Fl(\vec{r},n) = \Fl(\vec{r}',n)$ where $\vec{r}'$ is obtained by removing repeated elements in $\vec{r}$. If $\sQ$ and $\sQ'$ are $(r,n)$--matroids such that $\sQ'$ is a quotient $\sQ$, then $\sQ = \sQ'$. Thus, a $(\vec{r},n)$--flag matroid $\vec{\sQ}$ is canonically identified with a $(\vec{r}',n)$--flag matroid $\vec{\sQ}'$ obtained by removing repeated matroids, and together with the previous observation,  $\Fl(\vec{\sQ})$ is canonically identified with $\Fl(\vec{\sQ}')$. 
\end{remark}

Recall that the diagonal torus $H<\PGL(n)$ acts on $\Fl(\vec{r},n)$ as in Formula \eqref{eq:HonFl}. In fact, this action restricts to each stratum $\Fl(\vec{\sQ})$, and we can understand this scheme-theoretically in the following way. Consider the $N = \Z^{n} / \Zone$ grading 
	\begin{equation}
	\label{eq:Ndeg}
	    \deg_{N}(p_{\lambda}) = \sum_{i\in \lambda} \epsilon_{i}
	\end{equation}
	The following proposition is clear given the quadric generators $f_{\mu,\nu}(\vec{\sQ})$ of $I_{\vec{\sQ}}$ in Formula \eqref{eq:flagMatroidIdealGens}.
	\begin{proposition}
	\label{prop:LHomogeneous}
	    The ideal $I_{\vec{\sQ}}$ is homogeneous with respect to the $N$-grading.
	\end{proposition}

	\subsection{Inverse limits of flag matroid strata}
	\label{sec:inverseLimitsAndInitialDegenerations}
    Given a $(r,n)$--matroid $\sQ$, the tropicalization of $\Gr(\sQ)$ is denoted by $\TGr(\sQ) \subset N(\sQ)_{\R}$. Similarly, given a flag matroid $\vec{\sQ} = (\sQ_1,\ldots,\sQ_s)$, the tropicalization of $\Fl(\vec{\sQ})$ is denoted by
    \begin{equation*}
    \TFl(\vec{\sQ}) \subset N(\sQ_1)_{\R} \times \cdots \times N(\sQ_s)_{\R}.
    \end{equation*}
    We have inclusions $\TGr(\sQ) \subset \Dr(\sQ)$ \cite{HerrmannJensenJoswigSturmfels, Speyer2008} and $\TFl(\vec{\sQ}) \subset \Dr(\vec{\sQ})$ \cite{BrandtEurZhang}. In general this inclusion is strict, but we do have an equality in some cases, e.g., when $n\leq 5$ and $\vec{\sQ}$ is the uniform $((1,\ldots,n-1),n)$--flag matroid. Both the Gr\"obner fan on $\TGr(\vec{\sQ})$ and the fiber fan on $\Dr(\vec{\sQ})$ are preserved under translation by the linear subspace $L_{\R}$ (see Formula \eqref{eq:definingL}) by Proposition \ref{prop:LHomogeneous} and the discussion in \S\ref{sec:weightedPointConfigurations}, respectively.
    
    Define a partial order on the set of $(\vec{r},n)$--flag matroids by \label{notn:faceOrder} $\vec{\sP} \leq \vec{\sQ}$ if
    $\Delta(\vec{\sP})$ is a face of $\Delta(\vec{\sQ})$, i.e., there is a $\sv \in N_{\R}$ such that $\vec{\sP} = \vec{\sQ}_{\sv} = ((\sQ_{1})_{\sv}, \ldots, (\sQ_{s})_{\sv})$. 
    Given $\vec{\sw}\in \Dr(\vec{\sQ})$ and $\vecomega = (\vec{\sQ}, \vec{\sa}, \vec{\sw})$, the above partial order turns $\pQ(\vecomega)$ into a poset. The following is the key proposition that allows us to define the face morphisms $\Fl(\vec{\sQ}) \to \Fl(\vec{\sP})$ when $\vec{\sP} \leq \vec{\sQ}$, and the closed immersion in Theorem \ref{thm:closedimmersionintro}.
 
	\begin{proposition}
		\label{prop:init2strata}
		Let $\vec{\sw} = (\sw_1,\ldots,\sw_s) \in \TFl(\vec{\sQ})$, $\sv\in N_{\R}$, and set $\vec{\sP} = \vec{\sQ}_{\sv}^{\vec{\sw}}$. The inclusion $B_{\vec{\sP}} \subset B_{\vec{\sQ}}$ induces a morphism $\varphi_{\omega,\sv}: \init_{\vec{\sw}}\Fl(\vec{\sQ}) \to \Fl(\vec{\sP})$. 
	\end{proposition}
	
	\begin{proof}
	    With $\vec{\sP} = (\sP_{1},\ldots, \sP_{s})$, we have that $\sP_i = (\sQ_i)_{\sv}^{\sw_i}$, thus the bases of $\sP_i$ can be determined by Formula \eqref{eq:Qvw}. It suffices to show that the extension of $I_{\vec{\sP}}$ to $B_{\vec{\sQ}}$ is contained in $\init_{\vec{\sw}}I_{\vec{\sQ}}$. We use the quadric generators in Formula \eqref{eq:flagMatroidIdealGens}. 
		
		Suppose $f_{\mu,\nu}(\vec{\sP})$ is not identically 0, i.e., there is an $k_0\in \nu\setminus \mu$ such that $\mu\cup k_0 \in \sQ_i$ and $\nu\setminus  k_0 \in \sQ_j$. We must show that $f_{\mu,\nu}(\vec{\sP}) = \init_{\vec{\sw}} f_{\mu,\nu}(\vec{\sQ})$, i.e., if $k\in \nu\setminus \mu$, then $p_{\mu\cup k}p_{\nu\setminus k}$ is a monomial appearing in $f_{\mu,\nu}(\vec{\sP})$ if and only if it is a monomial in $\init_{\vec{\sw}} f_{\mu,\nu}(\vec{\sQ})$. The monomial $p_{\mu\cup k}p_{\nu\setminus k}$ is a monomial of $\init_{\vec{\sw}} f_{\mu,\nu}(\vec{\sQ})$ if and only if 
		\begin{equation}
		\label{eq:minimizew}
		\sw_{\mu\cup k} + \sw_{\nu\setminus k} \leq \sw_{\mu\cup \ell} + \sw_{\nu\setminus \ell}  \hspace{20pt} \text{for all } \ell.
		\end{equation}
		Next, observe that for any $\ell$ we have
		\begin{equation*}
		\bk{\epsilon_{\mu\cup \ell}^{*}}{\sv} + \bk{\epsilon_{\nu\setminus \ell}^{*}}{\sv} = \bk{\epsilon_{\mu}^{*}}{\sv} + \bk{\epsilon_{\nu}^{*}}{\sv} 
		\end{equation*}
		which is independent of $\ell$. Thus Formula \eqref{eq:minimizew} holds if and only if
		\begin{equation*}
		\bk{\epsilon_{\mu\cup k}^{*}}{\sv} + \sw_{\mu\cup k} + \bk{\epsilon_{\nu\setminus k}^{*}}{\sv} + \sw_{\nu\setminus k} \leq 
		\bk{\epsilon_{\mu\cup \ell}^{*}}{\sv} + \sw_{\mu\cup \ell} + \bk{\epsilon_{\nu\setminus \ell}^{*}}{\sv} + \sw_{\nu\setminus \ell}
		\end{equation*}
		for all $\ell$. This is true if and only if
		\begin{equation*}
		\bk{\epsilon_{\mu\cup k}^{*}}{\sv} + \sw_{\mu\cup k} + 
		\bk{\epsilon_{\nu\setminus k}^{*}}{\sv} + \sw_{\nu\setminus k} = 
		\bk{\epsilon_{\mu\cup k_0}^{*}}{\sv} + \sw_{\mu\cup k_0} + 
		\bk{\epsilon_{\nu\setminus k_0}^{*}}{\sv} + \sw_{\nu\setminus k_0}
		\end{equation*}
		which is true if and only if $\mu\cup k \in \sP_i$  and $\nu\setminus k \in \sP_j$, as required. 
	\end{proof}
	
    If $\vec{\sP} \leq \vec{\sQ}$, then by Proposition \ref{prop:init2strata} applied to the case $\vec{\sw} = 0$, we have a face morphism
	\begin{equation}
	\label{eq:faceMap}
	\varphi_{\vec{\sQ},\vec{\sP}} = \Fl(\vec{\sQ}) \to \Fl(\vec{\sP}). 
	\end{equation}
	These face morphisms satisfy
	\begin{equation}
	\varphi_{\vec{\sQ}_1,\vec{\sQ}_3} = \varphi_{\vec{\sQ}_2,\vec{\sQ}_3} \circ \varphi_{\vec{\sQ}_1,\vec{\sQ}_2} \; \text{ if } \; \vec{\sQ}_3 \leq \vec{\sQ}_2 \leq \vec{\sQ}_1, \hspace{10pt} \text{ and } \hspace{10pt} \varphi_{\vec{\sQ},\vec{\sQ}} = \id.
	\end{equation}
	This means that
	\begin{equation*}
	\vec{\sP} \mapsto \Fl(\vec{\sP}) \hspace{20pt} \vec{\sP'} \leq \vec{\sP} \mapsto \varphi_{\vec{\sP},\vec{\sP}'} 
	\end{equation*}
	defines a diagram of type $\pQ(\vecomega)$. Thus we may form the inverse limit
	\begin{equation*}
	\label{notn:Flomega}
	\Fl(\vecomega) := \varprojlim_{\pQ(\vecomega)} \Fl. 
	\end{equation*}
	This is an affine scheme whose coordinate ring is the $\GG_m^s$-invariants subring of $\varinjlim_{\pQ(\vecomega)} R_{\vec{\sP}}$. 
	
	\begin{theorem}
	\label{thm:closedimmersion}
		Let $\vec{\sw} \in \TFl(\vec{\sQ})$ and $\vecomega = (\vec{\sQ}, \vec{\sa}, \vec{\sw})$. The morphisms $\varphi_{\vecomega,\sv}:\init_{\vec{\sw}} \Fl(\vec{\sQ}) \to \Fl(\vec{\sQ}_{\sv}^{\vec{\sw}})$ induce a closed immersion
		\begin{equation}
		\varphi_{\vecomega}: \init_{\vec{\sw}}\Fl(\vec{\sQ}) \hookrightarrow \Fl(\vecomega). 
		\end{equation}
	\end{theorem}
	
	\begin{proof}
		The map $\varphi_{\vecomega}$ exists by the universal property. Given any $\lambda \in \pB(\vec{\sQ})$, there is a $\vec{\sP}$ defining a face of $\pQ(\vecomega)$ such that $\lambda \in \pB(\vec{\sP})$. Therefore, the induced map
		\begin{equation*}
		(\varphi_{\vecomega})^{\#}:\varinjlim_{\pQ(\vecomega)} R_{\vec{\sP}} \to S_{\vec{\sQ}}^{-1}B_{\vec{\sQ}} / \init_{\vec{\sw}} I_{\vec{\sQ}}
		\end{equation*}
		is surjective, and so $\varphi_{\vecomega}$ is a closed immersion. 
	\end{proof}
	
	\section{Faces of flag matroid polytopes}
	\label{sec:flagFaces}
	Let $\sQ$ be a $(r,n)$--matroid and $\Delta(\sQ)$ its polytope. The faces of $\Delta(\sQ)$ may be described as follows. Given $\sv \in N_{\R}$, set  $\sQ_{\sv} = \sQ_{\sv}^{0}$. The face of $\Delta(\sQ)$ minimized by the vector $\sv$ is $\Delta(\sQ_{\sv})$, and when $\sv= - \epsilon_{\lambda}$ we have 
	\begin{equation*}
		\sQ_{\sv} \cong (\sQ | \lambda) \oplus (\sQ/\lambda)
	\end{equation*}
	 where $\sQ | \lambda$ is the \textit{restriction} of $\sQ$ to $\lambda$, $\sQ/\lambda$ is the \textit{contraction} of $\sQ$ by $\lambda$ and $\oplus$ is the direct sum of matroids. 
	 The direct sum of matroids interacts well with the formation of matroid polytopes and matroid strata:
	 \begin{equation*}
	 \Delta(\sQ_1 \oplus \sQ_2) \cong \Delta(\sQ_{1}) \times \Delta(\sQ_{2}) \hspace{20pt} \text{ and } \hspace{20pt} \Gr(\sQ_1\oplus \sQ_2) \cong \Gr(\sQ_1) \times \Gr(\sQ_2). 
	 \end{equation*}

	 In this section, we recall how this generalizes to flag matroids. Our arguments use the \textit{Bergman fan} $\Berg(\sQ)$ of a $(r,n)$--matroid $\sQ$ \cite{ArdilaKlivans,Sturmfels2002}. 
	 If $\sQ$ is loopless, its Bergman fan $\Berg(\sQ)$ is the support of a fan in $N_{\R}$ whose cones are $\R_{\geq 0} \langle \epsilon_{\lambda_1},\ldots,\epsilon_{\lambda_k} \rangle$ for each 
 	 sequence $\lambda_1 \subsetneq \cdots \subsetneq \lambda_k$ of \textit{proper} flats, i.e., flats whose rank is positive and strictly less than the rank of $\sQ$. Equivalently, $\Berg(\sQ)$ is
 	 \begin{equation*}
 	     \Berg(\sQ) = \{ \sv \in N_{\R} \, : \, \sQ_{\sv} \text{ is loopless}  \}.
 	 \end{equation*}
 	 In general, let $\lp(\sQ)\subset [n]$ be the set of loops of $\sQ$ and $\widetilde{\sQ} = \sQ|([n]\setminus \lp(\sQ))$. Then $\Berg(\sQ)$ is defined to be $\Berg(\widetilde{\sQ})$, which is this is the support of a fan in $\widetilde{N}_{\R}^{\sQ}$ where  $\widetilde{N}^{\sQ} := N/(\Zone + \Z^{\lp(\sQ)})$.  The essential facts we need are contained in the following proposition; these are well-known when the relevant matroids are loopless, and one readily extends them to the general case.  

\begin{proposition}
\label{prop:bergman}
\leavevmode
    \begin{enumerate}
    \item If $\sQ$ is a matroid on $[n]$ and $\lambda \subset [n]$,  then $\Berg(\sQ|\lambda)$ equals the image of  $\Berg(\sQ)$ under the coordinate projection $\widetilde{N}_{\R}^{\sQ} \to \widetilde{N}_{\R}^{\sQ|\lambda}$. 
    
        \item The matroid $\sQ_1$ is a quotient of $\sQ_2$ if and only if $\lp(\sQ_2) \subset \lp(\sQ_1)$ and  $\Berg(\sQ_1) \subset p(\Berg(\sQ_2))$, where $p:\widetilde{N}^{\sQ_2} \to \widetilde{N}^{\sQ_1}$ is the coordinate projection.
         
        \item $\Berg(\sQ_1 \oplus \sQ_2) \cong \Berg(\sQ_1) \times \Berg(\sQ_2)$. 
    \end{enumerate}
\end{proposition}

    Given a $(\vec{r},n)$--flag matroid $\vec{\sQ} = (\sQ_1,\dots,\sQ_s)$ and $\lambda\subset [n]$, the \textit{restriction} $\vec{\sQ} | \lambda$ and \textit{contraction} $\vec{\sQ}/\lambda$ are defined componentwise:
	\begin{equation*}
	\vec{\sQ}|\lambda = (\sQ_1|\lambda,\ldots,\sQ_{s}|\lambda), \hspace{20pt} \vec{\sQ}/\lambda = (\sQ_1/\lambda,\ldots,\sQ_{s}/\lambda).
	\end{equation*}
	
	\begin{proposition}
	\label{prop:flagMatroidResCon}
		Given a flag matroid $\vec{\sQ}$ and $\lambda\subset [n]$, the restriction $\vec{\sQ}|\lambda$ and contraction $\vec{\sQ}/\lambda$ are flag matroids. 
	\end{proposition}

	\begin{proof}
	This proposition is \cite[Theorem~4.1.5]{BrandtEurZhang}, equivalently, \cite[Theorem~6.1]{AguiarArdila} and Formula \eqref{eq:facesMS}.
    We argue using the Bergman fan of a matroid. Consider the commutative diagram of coordinate projections
    \begin{equation*}
    \begin{tikzcd}  
    \widetilde{N}_{\R}^{\sQ_j} \arrow[r,"p_1"] \arrow[d, ,"p_2"']   
    & \widetilde{N}_{\R}^{\sQ_i} \arrow[d, ,"p_3"] \\
    \widetilde{N}_{\R}^{\sQ_j|\lambda} \arrow[r, ,"p_4"]
    &  \widetilde{N}_{\R}^{\sQ_i|\lambda}
    \end{tikzcd}
    \end{equation*}
    where $i\leq j$. By Proposition \ref{prop:bergman}, we have that $\Berg(\sQ_i) \subset p_1(\Berg(\sQ_j))$ and 
    \begin{equation*}
        p_4(\Berg(\sQ_j|\lambda)) = (p_4\circ p_2)(\Berg(\sQ_j))  \supseteq  p_3(\Berg(\sQ_i)) = \Berg(\sQ_i|\lambda). 
    \end{equation*}
    So $\sQ_i|\lambda$ is a quotient of $\sQ_j|\lambda$. This proves that $\vec{\sQ}|\lambda$ is a flag matroid.  Next consider contraction. The dual of $\vec{\sQ}$, defined by $\vec{\sQ}^* := (\sQ_s^{*}, \ldots, \sQ_1^{*})$, is a flag matroid   \cite[Proposition~7.4.7]{White}. By the restriction case, we have that $\vec{\sQ}^*|([n]\setminus \lambda)$ is a flag matroid, and therefore its dual, which is exactly $\vec{\sQ}/\lambda$, is also a flag matroid. 
	\end{proof}

	Before considering direct sums of flag matroids, we review direct sums of point configurations. Let $M$ and $M'$ be the lattices dual to $\Z^{n}/\Zone$ and  $\Z^{n'}/\Zone$, respectively, and let $(\sQ,\sa)$ and $(\sQ',\sa')$ be point configurations in $M$ and $M'$, respectively.  The direct sum  $(\sQ,\sa)\oplus (\sQ',\sa')$ is the point configuration $(\sQ\oplus \sQ', \sa \oplus \sa')$ in $M\oplus M'$, where
	\begin{equation*}
	    \sa \oplus \sa': \sQ\oplus \sQ' \to M\oplus M' \hspace{20pt} (\lambda,\lambda') \mapsto (\sa(\lambda),\sa'(\lambda'))
	\end{equation*}

    Given $(\vec{r},n)$ and $(\vec{r}',n')$--flag matroids $\vec{\sQ} = (\sQ_1,\ldots,\sQ_s)$ and $\vec{\sQ}'=(\sQ_{1}',\ldots,\sQ_{s}')$, their \textit{direct sum} is
    \begin{equation*}
    \vec{\sQ} \oplus \vec{\sQ}' = (\sQ_1\oplus \sQ_1' , \ldots, \sQ_s\oplus \sQ_s').
    \end{equation*} 
	\begin{proposition}
	\label{prop:flagStrataDirectSum}
		The direct sum $\vec{\sQ} \oplus \vec{\sQ}'$ is a $(\vec{r}+\vec{r}',n+n')$ matroid. Its polytope and stratum are
		\begin{equation*}
		    \Delta(\vec{\sQ} \oplus \vec{\sQ}') \cong \Delta(\vec{\sQ}) \times \Delta(\vec{\sQ}')  \hspace{20pt} \text{ and } \hspace{20pt} \Fl(\vec{\sQ} \oplus \vec{\sQ}') \cong \Fl(\vec{\sQ}) \times \Fl(\vec{\sQ}')
		\end{equation*}
	\end{proposition}
	
	\begin{proof}
	We argue using the Bergman fan. By Proposition \ref{prop:bergman}(1), we have that $\Berg(\sQ_i) \subset \Berg(\sQ_j)$ and $\Berg(\sQ_i') \subset \Berg(\sQ_j')$ for $i\leq j$. By part (3) of that proposition, we also have that
	$\Berg(\sQ_i \oplus \sQ_i') = \Berg(\sQ_i) \times \Berg(\sQ_i')$. So $\Berg(\sQ_i \oplus \sQ_i') \subset \Berg(\sQ_j \oplus \sQ_j')$ for $i\leq j$, and therefore $\vec{\sQ}\oplus \vec{\sQ}'$ is a flag matroid.

	Next, consider the point configurations $(\vec{\sQ},\vec{\sa})$ and  $(\vec{\sQ},\vec{\sa})$ associated to the flag matroids. We have
	\begin{equation*}
	    \MS((\vec{\sQ},\vec{\sa}) \oplus (\vec{\sQ}',\vec{\sa}') ) = \MS(\vec{\sQ},\vec{\sa}) \oplus \MS(\vec{\sQ}',\vec{\sa}'),
	\end{equation*}
	from which the statement on flag matroid polytopes follows. 

    Finally, consider the claim on flag strata.  This result is not critical for our main results, so we prove this only at the level of $\C$-points as it is clearer than a scheme-theoretic argument. Define
    \begin{align*}
        \Fl(\vec{\sQ}) \oplus \Fl(\vec{\sQ}') &\to \Fl(\vec{\sQ} \oplus \vec{\sQ}') \\
        ((F_{1}\subset \cdots \subset F_{s}),(F_{1}'\subset \cdots \subset F_{s}')) &\mapsto (F_{1}\oplus F_{1}'\subset \cdots \subset F_{s}\oplus F_{s}').
    \end{align*}
    Because  $\Gr(\sQ\oplus \sQ') = \Gr(\sQ) \times \Gr(\sQ')$, any flag  in $\Fl(\vec{\sQ}\oplus \vec{\sQ})$ is of the form $(F_{1}\oplus F_{1}'\subset \cdots \subset F_{s}\oplus F_{s}')$. We can use this to produce an inverse to the above map.
	\end{proof}
	
	Given a flag matroid $\vec{\sQ}$, any face of $\Delta(\vec{\sQ})$ is of the form $\Delta(\vec{\sQ}_{\sv})$, where $\sv=-\epsilon_{\lambda}$ for some $\lambda\subset [n]$; $\vec{\sQ}_{\sv}$ is the flag matroid $\vec{\sQ}_{\sv}^{0}$ from Formula \eqref{eq:facesMS}.
	
	\begin{proposition}
	\label{prop:faceRestrictionContraction}
	    Let $\vec{\sQ}$ be a flag matroid and $\sv=-\epsilon_{\lambda}$. Then the flag matroid $\vec{\sQ}_{\sv}$ is
	    \begin{equation*}
	        \vec{\sQ}_{\sv} = \vec{\sQ}|\lambda \oplus \vec{\sQ}/\lambda. 
	    \end{equation*}
	\end{proposition}
	\begin{proof}
	    This is a special case of  \cite[Theorem~6.1]{AguiarArdila}. Alternatively, observe that
	    \begin{equation*}
	        \vec{\sQ}_{\sv} = ((\sQ_1)_{\sv},\ldots,(\sQ_{s})_{\sv}) = (\sQ_1|\lambda \oplus \sQ_1/\lambda,\ldots,\sQ_1|\lambda \oplus \sQ_s/\lambda) = \vec{\sQ}|\lambda \oplus \vec{\sQ}/\lambda. \qedhere
	    \end{equation*}
	\end{proof}
	
	\begin{remark}
	The flag matroids $\vec{\sQ}|\lambda$ and $\vec{\sQ}/\lambda$ may have repeated constituents; see Remark \ref{rmk:repeatedRanks} on the formation of their strata. 
	\end{remark}

	\section{The complete flag variety}
	\label{sec:completeFlag}
	
	In the remaining sections, we specialize to the case of the complete flag variety
	\begin{equation*}
	\label{notn:completeFlag}
	\Fl(n) := \Fl((1,\ldots,n-1),n).    
	\end{equation*}
	
	\subsection{Affine coordinates}
	\label{sec:affineCoordinates}
	The \textit{standard flag} is the flag of subsets
	\begin{equation}
	\label{eq:standardFlag}
	\{1\} \subset \{1,2\} \subset \cdots \subset \{1,\ldots,n-1\}.
	\end{equation}
	Let $\vec{\sQ}$ be a $\C$-realizable flag matroid of rank $(1,\ldots,n-1)$ on $[n]$, and without loss of generality, assume that $\pB(\vec{\sQ})$ contains the standard flag. Define the polynomial ring 
	\begin{equation*}
	B^{x} = \C[x_{ij} \, : \, 1\leq i < j \leq n]
	\end{equation*}
	and $B^{x}$-valued matrix
	\begin{equation}
	\label{eq:matrixX}
	X = \begin{bmatrix}
	1 & x_{12} & \cdots & x_{1,n-1} & x_{1n} \\
	0 & 1 & \cdots & x_{2,n-1} & x_{2n} \\
	\cdots & \cdots & \cdots & \cdots & \cdots \\
	0 & 0 & \cdots & 1 & x_{n-1,n} \\
	0 & 0 & \cdots & 0 & 1 
	\end{bmatrix}
	\end{equation}
	Given $\lambda \subset [n]$, let $X_{\lambda}$ denote the determinant of the submatrix of $X$ whose rows are $1,\ldots,|\lambda|$ and columns are indexed by $\lambda$. Define
	\begin{equation*}
	B_{\vec{\sQ}}^{x} = \C[x_{ij} \, : \, [i-1]\cup j \in \pB(\vec{\sQ})]   
	\end{equation*}
	and the quotient map
	\begin{equation*}
	\pi_{\vec{\sQ}}: B^x \onto B^{x} / \langle x_{ij} \, : \, [i-1]\cup j \notin \pB(\vec{\sQ}) \rangle \xrightarrow{\sim}  B_{\vec{\sQ}}^x.
	\end{equation*}
	Now let
	\begin{equation*}
	I_{\vec{\sQ}}^{x} = \langle \pi_{\vec{\sQ}}(X_{\lambda}) \, : \, \lambda \notin \pB(\vec{\sQ}) \rangle \hspace{20pt}  S_{\vec{\sQ}}^{x} = \langle \pi_{\vec{\sQ}}(X_{\lambda}) \, : \, \lambda \in \pB(\vec{\sQ}) \rangle_{\semigp}.
	\end{equation*}
	and set 
	\begin{equation*}
	R_{\vec{\sQ}}^{x} = (S_{\vec{\sQ}}^{x})^{-1} B_{\vec{\sQ}}^{x} / I_{\vec{\sQ}}^{x}.
	\end{equation*}
	The ring homomorphism
	\begin{equation*}
	\psi_{\vec{\sQ}} : R_{\vec{\sQ}}^{x} \to (R_{\vec{\sQ}})^{\GG_m^s} \hspace{30pt}
	x_{ij} \mapsto p_{[i-1] \cup j} / p_{[i]}.
	\end{equation*}
	defines an isomorphism of $R_{\vec{\sQ}}^{x}$ to the coordinate ring of $\Fl(\vec{\sQ})$. 
	
	\begin{proposition}
		\label{prop:mapxcoordinates}
		The morphism $\varphi_{\vec{\sQ},\vec{\sP}'}: \Fl(\vec{\sQ}) \to \Fl(\vec{\sP}')$ corresponds to the ring homomorphism 
		\begin{equation*}
		\varphi_{\vec{\sQ},\vec{\sP}'}^{\#}: R_{\vec{\sP}'}^{x} \to R_{\vec{\sQ}}^{x} \hspace{20pt} x_{ij} \mapsto x_{ij}.
		\end{equation*}
	\end{proposition}
	
	\begin{proof}
		As usual, assume that $\pB(\vec{\sQ})$ and $\pB(\vec{\sP}')$ contain the standard flag. The composition $\psi_{\vec{\sQ}}^{-1} \circ \varphi_{\vec{\sQ},\vec{\sP}'}^{\#} \circ \psi_{\vec{\sP}'}$ sends $x_{ij}$ to $x_{ij}$, as required. 
	\end{proof}

	We illustrate these constructions with the following example, which we also use in the proof of Theorem \ref{thm:initFl4}.
	\begin{example}
	\label{ex:forClosedImmersion}
		Let $\vec{\sQ}$ be the 4th flag matroid in Table \ref{table:Flag123-4}, i.e., the one with bases
		\begin{equation*}
		\pB(\vec{\sQ}) = \{1,2,4,12,24,123,124,234\}.
		\end{equation*}
		To simplify notation, let $u=x_{12}$, $v=x_{13}$, $w=x_{14}$, $x=x_{23}$, $y=x_{24}$, $z=x_{34}$. Because $2,4,124 \in \pB(\vec{\sQ})$ and $3,13,14 \notin \pB(\vec{\sQ})$, we have that $B_{\vec{\sQ}}^{x} = \C[u,w,z]$. Furthermore,  $\pi_{\vec{\sQ}}(X_{\lambda}) = \tilde{X}_{\lambda}$ where
		\begin{equation*}
		\tilde{X} = \begin{bmatrix}
		1 & u & 0 & w \\
		0 & 1 & 0 & 0 \\
		0 & 0 & 1 & z \\
		0 & 0 & 0 & 1 
		\end{bmatrix}
		\end{equation*}
		so $I_{\vec{\sQ}}^{x} = \langle 0 \rangle$ and $S_{\vec{\sQ}}^{x} = \langle u,w,z \rangle_{\semigp}$. Therefore $R_{\vec{\sQ}}^{x} = \C[u^{\pm},w^{\pm}, z^{\pm}]$. 
		
		Now consider the two flag matroids $\vec{\sQ}_a, \vec{\sQ}_b \leq \vec{\sQ}$ defined by
		\begin{equation*}
		\pB(\vec{\sQ}_{a}) = \{1,2,4,12,24,123,234\}, \hspace{20pt} \pB(\vec{\sQ}_{b}) = \{1,4,12,24,123,124,234\}.
		\end{equation*}
		(These are \textit{internal} in the sense of \S \ref{sec:flags4}.)
		By a similar computation, we see that
		\begin{equation*}
		R_{\vec{\sQ}_{a}}^{x} = \C[u^{\pm},w^{\pm}], \hspace{20pt} R_{\vec{\sQ}_{b}}^x = \C[w^{\pm},z^{\pm}].
		\end{equation*}
		Using Proposition \ref{prop:mapxcoordinates}, $\Fl(\vec{\sQ}) \to \Fl(\vec{\sQ}_a)$ and $\Fl(\vec{\sQ}) \to \Fl(\vec{\sQ}_b)$ may be identified with coordinate projections of tori of the form $\GG_m^3 \to \GG_m^2$, which are smooth and surjective with connected fibers. Also, the morphism
		\begin{equation}
		(\varphi_{\vec{\sQ},\vec{\sQ}_a}, \varphi_{\vec{\sQ},\vec{\sQ}_b}): \Fl(\vec{\sQ}) \to \Fl(\vec{\sQ}_a) \times \Fl(\vec{\sQ}_b) 
		\end{equation}
		is a closed immersion since the induced map on rings 
		\begin{equation*}
		\C[u^{\pm},z^{\pm}] \otimes_{\C} \C[u^{\pm},w^{\pm}] \to \C[u^{\pm},w^{\pm},z^{\pm}]
		\end{equation*}
		is surjective. We will use this fact in the proof of Theorem \ref{thm:initFl4}. 
	\end{example}

	\begin{proposition}
		\label{prop:morphismsTrick}
		Consider two flag matroids $\vec{\sP} \leq \vec{\sQ}$ with $\dim \Fl(\vec{\sQ}) = d$. Let 
		\begin{equation*}
		\emptyset = \lambda_0 \subsetneq \lambda_1 \subsetneq \cdots \subsetneq \lambda_{n-1} 
		\end{equation*}
		be a complete flag such that $\{\lambda_1,\ldots,\lambda_{n-1}\} \subset \pB(\vec{\sP})$ (and hence a subset of $\pB(\vec{\sQ})$).
		If
		\begin{equation*}
		| \{ i  \, : \,  i=0,\ldots,n-1;\; a\notin \lambda_{i+1};\; \lambda_i \cup a \in \sQ_{i+1}  \} | = d
		\end{equation*}
		then $\Fl(\vec{\sP})$ and $\Fl(\vec{\sQ})$ are smooth and irreducible, and $\Fl(\vec{\sQ}) \to \Fl(\vec{\sP})$ is smooth and dominant with connected fibers. 
	\end{proposition}
	
	\begin{proof}
		After applying a suitable permutation, we may assume that $\lambda_r = [r]$. By the hypothesis, we have that $\dim B_{\vec{\sQ}}^{x} = d$, and therefore $I_{\vec{\sQ}}^{x} = \langle 0 \rangle$. By Proposition \ref{prop:mapxcoordinates}, we also have that $I_{\vec{\sP}}^{x} = \langle 0 \rangle$. This means that $\Fl(\vec{\sQ})$ and $\Fl(\vec{\sP})$ are isomorphic to open subvarieties of tori and the morphism $\Fl(\vec{\sQ}) \to \Fl(\vec{\sP})$ is induced by a coordinate projection. 
	\end{proof}

	\subsection{Symmetry of $\Flo{n}$}
	Denote by \label{notn:Sn} $\Sn{n}$ the symmetric group on $[n]$. 
	The group $\Sn{n}$ acts on  $\Flo{n}$ by permuting the coordinates of $\C^{n}$, and $\Sn{2}$ acts on $\Flo{n}$ by duality, i.e., the unique generator acts on a complete flag by
	\begin{equation*}
	F_1\subset \cdots \subset F_{n-1} \mapsto F_{n-1}^{\perp} \subset \cdots \subset F_{1}^{\perp} 
	\end{equation*}
	These two actions commute, so we have an action of  of $\Sn{2} \times \Sn{n}$ on $\Flo{n}$ and hence on $\TFlo{n}^{\circ}$.  This is the symmetry group of the permutahedron $\Pi_n$ for $n\geq 4$, see \cite[Proposition~2]{JoswigLohoLuberOlarte}.

	\subsection{Strata for $n=3$} 
	\label{sec:strata3}
	There are 5 rank-$(1,2)$ flag matroids on $[3]$ up to $(\Sn{2} \times \Sn{3})$--symmetry:
	\begin{gather*}
	\pB(\vec{\sQ}_1) = \{1,12\}, \hspace{20pt} \pB(\vec{\sQ}_2) = \{1,2,12\}, \hspace{20pt} \pB(\vec{\sQ}_3) = \{1,3,12,23\},  \\ 
	\pB(\vec{\sQ}_4) = \{1,2,3,12,23\}, \hspace{20pt} \pB(\vec{\sQ}_5) = \{1,2,3,12,13,23\}.  
	\end{gather*}
	Their strata are smooth and irreducible of dimensions $0,1,1,2,3$, respectively. 
	
	\subsection{Strata for $n=4$} \label{sec:flags4}
	Recall from Remark \ref{rmk:matroidPolytope} that, for $\vec{r} = (1,\ldots,n-1)$, the flag matroid polytope of the uniform $(\vec{r},n)$--matroid is the $(n-1)$--dimensional permutahedron. Given  $(\vec{r},n)$--flag matroids $\vec{\sP} \leq \vec{\sQ}$, we say that $\vec{\sP}$ is \textit{internal} if  $\Delta(\vec{\sP})$ meets the relative interior of $\Pi_n$. 
	
	There are, up to $(\Sn{2} \times \Sn{4})$--symmetry, 29 rank $(1,2,3)$--rank flag matroids on $[4]$. There are 15 such flag matroids whose polytope has the largest possible dimension, i.e., equal to $\dim(\Pi_4)  = 3$. These flag matroids $\vec{\sQ}$ are listed in Table \ref{table:Flag123-4}, together with their coordinate rings and internal $\vec{\sP} \leq \vec{\sQ}$. Each pair $\vec{\sP} \leq \vec{\sQ}$ is labeled by a number (left-most column) and a letter (directly adjacent to the set of nonbases of $\vec{\sP}$ in the third column).

	\begin{proposition}
		\label{prop:flagStrata4}
		The strata $\Fl(\vec{\sQ})$ for rank $(1,2,3)$ flag matroids on $[4]$ are smooth and irreducible. 
	\end{proposition}
	
	\begin{proof}
		If $\dim \Delta(\vec{\sQ}) < 3 = \dim(\Pi_4)$, then $\vec{\sQ} \cong \vec{\sQ}|_{\lambda} \times \vec{\sQ}/\lambda$  for some $\lambda \subseteq [n]$ by Proposition \ref{prop:faceRestrictionContraction}, so $\Fl(\vec{\sQ})$ is smooth and irreducible by Proposition \ref{prop:flagStrataDirectSum}, \S \ref{sec:strata3}, and smoothness and irreducibility of thin Schubert cells of matroids on fewer than $4$ elements. 
		
		Now consider the $\dim \Delta(\vec{\sQ}) = 3$ case.  We proceed as in \S \ref{sec:affineCoordinates}, but to simplify notation, we set $u=x_{12}$, $v=x_{13}$, $w=x_{14}$, $x=x_{23}$, $y=x_{24}$, $z=x_{34}$. For all but the 12th flag matroid in Table \ref{table:Flag123-4}, it is immediately clear that $\Fl(\vec{\sQ})$ is isomorphic to a dense open subvariety of an algebraic torus. If $\vec{\sQ}$ is the 12th flag matroid, then
		\begin{equation}
		\label{eq:mat12Iso}
		\vartheta: R_{\vec{\sQ}}^{x} \to \C[u^{\pm},v^{\pm},y^{\pm},z^{\pm}] \hspace{20pt} (u,v,w,y,z) \mapsto (u,v,uy,y,z)
		\end{equation}
		defines a ring isomorphism, the inverse is given by $(u,v,y,z)\mapsto (u,v,y,z)$. 
	\end{proof}

	\begin{proposition}
		\label{prop:flagStrataMaps4}
		If $\vec{\sQ}$ is a rank-$(1,2,3)$ flag matroid on $[4]$ such that $\dim \Delta(\vec{\sQ}) = 3$, and $\vec{\sP}\leq \vec{\sQ}$ is internal, then the morphism  $\Fl(\vec{\sQ})\to \Fl(\vec{\sP})$ is smooth and dominant with connected fibers.
	\end{proposition}

	\begin{proof}
		We may apply Proposition \ref{prop:morphismsTrick} to the following pairs $\vec{\sP} \leq \vec{\sQ}$, thus deducing that  $\Fl(\vec{\sQ}) \to \Fl(\vec{\sP})$ is smooth and dominant with connected fibers in these cases:
		\begin{itemize}
			\item[-] 1(a), 1(d), 2(a), 2(c), 3(a), 4(a), 4(b), 5(a), 6(a), 6(b), 6(c), 7(a), 7(c), 8(a), 9(b), 10(a), 10(b), 11(a), 11(b), 13(a), 14(a) (use the flag $\emptyset \subset 1\subset 12 \subset 123$);
			\item[-] 1(b), 1(c), 3(b) (use the flag  $\emptyset \subset 1 \subset 14 \subset 124$); 
			\item[-] 2(b)  (use the flag $\emptyset \subset 2 \subset 24\subset 234$); 
			\item[-]  7(b)  (use the flag $\emptyset \subset 3 \subset 34\subset 234$). 
		\end{itemize}
		The only pairs $\vec{\sP} \leq \vec{\sQ}$ remaining are 12(a) and 12(b); the permutation $(12)(34)$ (in cyclic notation) interchanges these two pairs, so it suffices to consider the pair $\vec{\sP} \leq \vec{\sQ}$ in 12(a). The composition of $(\varphi_{\vec{\sQ},\vec{\sP}})^{\#}$ with the isomorphism $\vartheta$ from Formula \ref{eq:mat12Iso} is the ring homomorphism
		\begin{equation*}
		R_{\vec{\sP}}^{x} = \C[v^{\pm}, y^{\pm}] \to \C[u^{\pm},v^{\pm},y^{\pm},z^{\pm}] \hspace{20pt} (v,y) \mapsto (v,y)
		\end{equation*}
		which defines a morphism of affine schemes that is smooth and surjective with connected fibers. 
	\end{proof}

	\section{Inverse limits of complete matroid strata}
	\label{sec:inverseLimitsCompleteFlag}
	
	Given a flag matroid $\vec{\sQ}$ with a weighted point configuration $\vecomega$, the subdivision $\sQ(\vecomega)$ has numerous cells, making the computation of the inverse limit $\pQ(\vecomega)$ \textit{a priori}  unwieldy. In this section, we describe a much smaller subposet of $\pQ(\vecomega)$ combinatorially equivalent to the  \textit{adjacency graph} of $\pQ(\vecomega)$, and the limit of the corresponding diagram of flag matroid strata coincides with $\Fl(\vecomega)$. We prove Theorem \ref{thm:Flag4intro} as an application of this characterization of $\Fl(\vecomega)$ and Theorem \ref{thm:closedimmersionintro}. 
	
	\subsection{Inverse limits over graphs}
	Let $\pC$ be a category that has finite limits, e.g., the category of affine $\C$-schemes. We recall the notion of an inverse limit in $\pC$ parameterized by a finite connected graph $G$ (possibly with loops or multiple edges) in the sense of  \cite[Appendix~A]{CoreyGrassmannians}.
	Define a quiver $\Quiv(G)$ as follows. The vertex set of $\Quiv(G)$ is the set $V(G) \cup E(G)$; write $q_v$, resp. $q_e$, the vertex of $Q(G)$ corresponding to the vertex $v$, resp. the edge $e$, of $G$. The quiver $\Quiv(G)$ has an arrow $q_v \to q_e$ for each half-edge half-edge of $e$ adjacent to $v$. Viewing $\Quiv(G)$ as a category in the usual way, a \textit{diagram} of type $\Quiv(G)$ in $\pC$ is a functor $X:\Quiv(G) \to \pC$. The inverse limit of $X$ parameterized by the quiver $\Quiv(G)$ is denoted by
	\begin{equation*}
	\varprojlim_{G} X.
	\end{equation*}

    Let $\omega = (\sQ,\sa,\sw)$ be a weighted point configuration. The \textit{adjacency graph} of $\pQ(\omega)$, denoted by $\Gamma(\omega)$, is the graph with a vertex $v_{\sP}$ for each maximal element $\sP$ of $\pQ(\omega)$, and two vertices $v_{\sP_1}, v_{\sP_2}$ share an edge whenever $\sP_1$ and $\sP_2$ share a common facet. Given  $\sR \in \pQ(\omega)$, denote by $\Gamma_{\sR}(\omega)$ the full subgraph of $\Gamma(\omega)$ on the vertices $v_{\sP}$ such that $\sR \leq \sP$. 
    
    Given a flag matroid $\vec{\sQ}$ and $\vecomega = (\vec{\sQ},\vec{\sa},\vec{\sw})$ with  $\vec{\sw} \in \Dr(\vec{\sQ})$, denote by $\Gamma(\vec{\omega})$ the adjacency graph of $\pQ(\vecomega)$, and given a maximal element $\vec{\sP}$ of $\pQ(\vecomega)$, denote by $v_{\vec{\sP}}$ the corresponding vertex. We now show that $\Gamma(\vecomega)$ provides sufficient data to compute $\Fl(\vecomega)$. Our arguments follow those of the analogous statements for limits of thin Schubert cells as in \cite[Appendix C]{CoreyGrassmannians} by Cueto, thus we only sketch the proofs of the following two propositions. 
    
    \begin{proposition}
    \label{prop:subdBasisConnected}
    Given a weighted point configuration $\omega=(\sQ,\sa,\sw)$ and any $\sR \in \pQ(\omega)$, the graph $\Gamma_{\sR}(\omega)$ is connected.  
    \end{proposition}
    
    \begin{proof}
        Without loss of generality, assume that $\dim \Delta(\sQ) = \dim M_{\R}$ (otherwise, replace $M_{\R}$ by the affine span of $\Delta(\sQ)$). Suppose $\sP_1$ and $\sP_2$  are two maximal elements of $\pQ(\omega)$ such that $\sR \leq \sP_1, \sP_2$. Let $y$ be any point in the relative interior of $\Delta(\sR)$, and $K$ be a neighborhood of $y$ chosen small enough so that every $x\in K$ lies in a cell $\Delta(\widetilde{\sR})$ with $\widetilde{\sR} \in \pQ(\omega)$ containing $\Delta(\sR)$ as a face.  Choose a line segment $[x_1,x_2]$ lying in $K$ such that $x_1\in P_1$, $x_2\in P_2$, and $[x_1,x_2]$ does not meet any codimension-2 face of $\pQ(\omega)$. Then $[x_1,x_2]$ meets only codimension 0 and 1 faces of $\pQ(\omega)$, each of which has $\sR$ as a face. By recording the cells met by traversing $[x_1,x_2]$, we get a path from $v_{\sP_1}$ to $v_{\sP_2}$ in $\Gamma_{\sR}(\omega)$, as required.
    \end{proof}

	\noindent 
	
	\begin{proposition}
		\label{prop:fullComplex2Graph}
		There is an isomorphism of inverse limits 
		\begin{equation*}
		\Fl(\vecomega) \cong \varprojlim_{\Gamma(\vecomega)} \Fl. 
		\end{equation*}
	\end{proposition}
	
    \begin{proof}
		Denote by $\Fl_{\Gamma}(\vecomega)$ the inverse limit on the right.  
		As the opposite quiver $\Quiv(\Gamma(\vecomega))^{\op}$ naturally corresponds to a subposet of $\pQ(\vecomega)$ (equipped with the face order), the universal property of inverse limits yields a morphism $\Phi:\Fl(\vecomega)  \to \Fl_{\Gamma}(\vecomega)$. 

		Next, we construct a morphism $\Psi:\Fl_{\Gamma}(\vecomega)  \to \Fl(\vecomega)$ whose inverse is $\Phi$. Suppose $\vec{\sR} \in \pQ(\vecomega)$. Given any maximal element $\vec{\sP}$ of $\pQ(\vecomega)$ dominating $\vec{\sR}$, define $\Fl_{\Gamma}(\vecomega) \to \Fl(\vec{\sR})$ by the composition
		\begin{equation*}
		    \Fl_{\Gamma}(\vecomega) \to \Fl(\vec{\sP}) \xrightarrow{\varphi_{\vec{\sP},\vec{\sR}}} \Fl(\vec{\sR}).
		\end{equation*}
	This map does not depend on the choice of maximal $\vec{P} \geq \vec{\sR}$ as the induced subgraph of  $\Gamma(\vecomega)$ on the vertices $v_{\vec{\sP}}$ where $\vec{\sP}\geq \vec{\sR}$  is connected
	by Proposition \ref{prop:subdBasisConnected}. 
	\end{proof}

	\subsection{Initial degenerations of $\Flo{3}$ and $\Flo{4}$}
	Recall that $\{\se_{\lambda} \, : \, \lambda \in \sQ\}$ are the images of the standard basis vectors under $\Z^{\sQ} \to N(\sQ)$, and when $\sQ=\binom{[n]}{r}$ we have a natural identification $N(\sQ) = \wedge^r\Z^n / \Zone$. First, consider $\Flo{3}$. The multihomogeneous ideal of $\Fl(3)$ as a subvariety of $\P(\wedge^1 \C^3) \times \P(\wedge^2\C^3)$ is 
	\begin{equation*}
	\langle p_{1}p_{23} - p_{2}p_{13} + p_{3}p_{12} \rangle \subset \C[p_{1},p_{2},p_{3}] \otimes \C[p_{12},p_{13},p_{23}].    
	\end{equation*}
	The tropicalization $\TFl^{\circ}(3)$ lies in $(\wedge^1 \R^3) / \Zone \times (\wedge^2 \R^3) / \Zone$. 
	Its lineality space is $L_{\R}$ where
	\begin{equation*}
	L = \Span_{\Z}\left\{ \begin{array}{ccc}
	     \se_{1} + \se_{12} + \se_{13},& \se_{2} + \se_{12} + \se_{23},& \se_{3} + \se_{13} + \se_{23},  \\
	     \se_{2} + \se_{3} + \se_{23},& \se_{1} + \se_{3} + \se_{13},& \se_{1} + \se_{2} + \se_{12}  \\
	\end{array}  \right\}
	\end{equation*}
	and $\TFl^{\circ}(3)$ has 4 cones:
	\begin{equation*}
	\begin{array}{cccc}
	L_{\R}, & \R_{\geq 0} \cdot \se_1 + L_{\R}, & \R_{\geq 0} \cdot \se_2 + L_{\R}, & \R_{\geq 0} \cdot \se_3 + L_{\R}.
	\end{array}
	\end{equation*}
	\begin{proposition}
	    The initial degenerations of $\Fl^{\circ}(3)$ are smooth and irreducible.
	\end{proposition}
    \begin{proof}
    Up to $(\Sn{2}\times\Sn{3})$--symmetry, there are only 2 distinct initial degenerations: $\Fl^{\circ}(3)$ and $\init_{\vec{\sw}}\Fl^{\circ}(3)$ where $\vec{\sw} = \se_1$; it suffices to consider the latter. Since the ideal of $\Fl(3)$ is principal, the coordinate ring of $\init_{\vec{\sw}} \Fl^{\circ}(3)$ is the subring of $\GG_m^2$-invariants of
    \begin{equation*}
        \C[p_{1}^{\pm},p_{2}^{\pm},p_{3}^{\pm}]\otimes \C[p_{12}^{\pm},p_{13}^{\pm},p_{23}^{\pm}] / \langle - p_{2}p_{13} + p_{3}p_{12}  \rangle
    \end{equation*}
    which defines a smooth and irreducible variety. To illustrate the application of Theorem \ref{thm:closedimmersion}, we sketch another argument. Let 
    \begin{equation*}
        \vec{\sQ} = (\sQ_1,\sQ_2) = (\{1,2,3\}, \{12,13,23\})
    \end{equation*}
    and $\vec{\omega}$ the weighted point configuration corresponding to $\vec{\sw} = \se_1$.     The subdivision $\pQ(\vecomega)$, illustrated in Figure \ref{fig:hexagon}, has 2 maximal faces, whose flag matroids are $\vec{\sQ}_a$, $\vec{\sQ}_b$, that meet along $\vec{\sQ}_{ab}$ where
    \begin{equation*}
        \vec{\sQ}_{a} = (\{1,2,3\}, \{12,13\}), \hspace{20pt} \vec{\sQ}_{a} = (\{2,3\}, \{12,13,23\}), \hspace{20pt} \vec{\sQ}_{ab} = (\{2,3\}, \{12,13\}). 
    \end{equation*}
 By Proposition \ref{prop:fullComplex2Graph}, we have $\Fl(\vecomega) \cong \Fl(\vec{\sQ}_a) \times_{\Fl(\vec{\sQ}_{ab})} \Fl(\vec{\sQ}_b)$.    The stratum $\Fl(\vec{\sQ}_{a})$ is smooth and irreducible of dimension $2$, and the morphism $\Fl(\vec{\sQ}_{b}) \to \Fl(\vec{\sQ}_{ab}) $ is smooth and surjective whose fibers are connected and 1-dimensional. By \cite[Proposition~A.2]{CoreyGrassmannians}, $\Fl(\vecomega)$ is smooth and irreducible of dimension 3, and therefore $\init_{\vec{\sw}}\Fl^{\circ}(3)$ is smooth and irreducible by Theorem \ref{thm:closedimmersion} and \cite[Proposition~A.8]{CoreyGrassmannians}.
    \end{proof}

	Next, consider $\Flo{4}$. The multihomogeneous ideal of $\Fl(4)$ as a subvariety of $\P(\wedge^1\C^4) \times \P(\wedge^2\C^4) \times \P(\wedge^3\C^4)$ is the ideal of 
	\begin{equation*}
	\C[p_{1},p_{2},p_{3},p_{4}] \otimes \C[p_{12},p_{13},p_{14},p_{23},p_{24},p_{34}]\otimes \C[p_{123},p_{124},p_{134},p_{234}]
	\end{equation*}
	generated by
	\begin{align*}
	\begin{array}{ll}
	p_{1}p_{23} - p_{2}p_{13} + p_{3}p_{12}, & p_{1}p_{24} - p_{2}p_{14} + p_{4}p_{12}, \\
	p_{1}p_{34} - p_{3}p_{14} + p_{4}p_{13}, & p_{2}p_{34} - p_{3}p_{24} + p_{4}p_{23}, \\
	p_{12}p_{134} - p_{13}p_{124} + p_{14}p_{123}, & p_{12}p_{234} - p_{23}p_{124} + p_{24}p_{123}, \\
	p_{13}p_{234} - p_{23}p_{134} + p_{34}p_{123}, & p_{14}p_{234} - p_{24}p_{134} + p_{34}p_{124}, \\
	p_{12}p_{34} - p_{13}p_{24} + p_{14}p_{23}, & p_{1}p_{234} - p_{2}p_{134} + p_{3}p_{124} - p_{4}p_{123}.
	\end{array}
	\end{align*}
	The tropicalization of $\Flo{4}$ originally appears in \cite{BossingerLambogliaMinchevaMohammadi}; we summarize the data here. The set $\TFl^{\circ}(4)$ lies in 
	\begin{equation*}
	(\wedge^1\R^4)/\Zone \times (\wedge^2\R^4)/\Zone \times (\wedge^3\R^4) /\Zone \cong \R^{14}/\R^{3}.
	\end{equation*}  
	Its lineality space $L_{\R}$ is 3-dimensional, and the saturated subgroup $L$ is generated by 
	\begin{align*}
	\begin{array}{ll}
	     \se_{1} + \se_{12} + \se_{13} + \se_{14} + \se_{123} + \se_{124} + \se_{134}, &\se_{2} + \se_{12} + \se_{23} + \se_{24} + \se_{123} + \se_{124} + \se_{234}, \\
	     \se_{3} + \se_{13} + \se_{23} + \se_{34} + \se_{123} + \se_{134} + \se_{234}, &	\se_{4} + \se_{14} + \se_{24} + \se_{34} + \se_{124} + \se_{134} + \se_{234}, \\
    	 \se_{2} + \se_{3} + \se_{4} + \se_{23} + \se_{24} + \se_{34} + \se_{234}, &\se_{1} + \se_{3} + \se_{4} + \se_{13} + \se_{14} + \se_{34} + \se_{134}, \\
	     \se_{1} + \se_{2} + \se_{4} + \se_{12} + \se_{14} + \se_{24} + \se_{124}, &	\se_{1} + \se_{2} + \se_{3} + \se_{12} + \se_{13} + \se_{23} + \se_{123}.
	\end{array}
	\end{align*}
	The $f$-vector of $\TFl^{\circ}(4)$ is 
	\begin{equation*}
	f(\TFlo{4}) = (1,20,79,78) \equiv (1,3,5,5) \mod \Sn{2} \times \Sn{4}.
	\end{equation*}
	and $(\Sn{2}\times \Sn{4})$--orbit representatives of the cones are contained in Table \ref{tab:cones123-4}.
	
	\begin{table}[tbh!]
		\centering
		\begin{tabular}{|c|c|c|c||c|c|c|c|}
			\hline 
			& \textbf{Ray generators} & \textbf{dim} & $\Sn{2}\times \Sn{4}$ & & \textbf{Ray generators} & \textbf{dim} & $\Sn{2}\times \Sn{4}$ \\
			& & &  \textbf{orbit size} &  &  & & \textbf{orbit size}  \\
			\hline 
			\hline
			\textbf{1} & $\se_1$ & 4 & 8 & \textbf{8} &$\se_{12}, \se_{34}$ & 5 & 3 \\ 
			\hline 
			\textbf{2} & $\se_{12}$ & 4 & 6 & \textbf{9} & $\se_{12}, \se_3 + \se_4 + \se_{34}$ & 5 & 12 \\ 
			\hline 
			\textbf{3} & $\se_1 + \se_2 + \se_{12}$ & 4 & 6 & \textbf{10} & $\se_{1}, \se_{23}, \se_{124}$ & 6 & 24 \\
			\hline 
			\textbf{4} & $\se_{1}, \se_{23}$ & 5 & 24 & \textbf{11} & $\se_{1}, \se_{23}, \se_{1}+\se_{4}+\se_{14}$ & 6 & 24 \\
			\hline 
			\textbf{5} & $\se_{1}, \se_{123}$ & 5 & 12 & \textbf{12} & $\se_{1}, \se_{123}, \se_{1}+\se_{4}+\se_{14}$ & 6 & 12 \\
			\hline 
			\textbf{6} & $\se_{1}, \se_{234}$ & 5 & 4 & \textbf{13} & $\se_1, \se_{234}, \se_{1}+\se_{2}+\se_{12}$ & 6 & 12 \\
			\hline 
			\textbf{7} & $\se_{1}, \se_{1}+\se_{2}+\se_{12}$ & 5 & 24 & \textbf{14} & $\se_{12}, \se_{34}, \se_{1}+\se_{2}+\se_{12}$ & 6 & 6 \\
			\hline 
		\end{tabular}
		\caption{Cones of $\TFl^{\circ}(4)$ up to $(\Sn{2}\times \Sn{4})$--symmetry}
		\label{tab:cones123-4}
	\end{table}

	\begin{theorem}
		\label{thm:initFl4}
		Let  $\vec{\sw} \in \TFl^{\circ}(4)$ and $\vecomega = (\vec{\sQ},\vec{\sa},\vec{\sw})$. The inverse limit $\Fl(\vecomega)$ is smooth and irreducible of dimension $6 = \dim \Fl(4)$. In particular, $\init_{\vec{\sw}}\Flo{4}$ is smooth and irreducible.
	\end{theorem}
	
	\begin{proof}
    If $\Gamma(\vecomega)$ is a tree, as in the graphs  1, 2, 4, 6, or 8, then $\varprojlim_{\Gamma(\vecomega)} \Fl$ is smooth and irreducible of dimension 6 by Propositions \ref{prop:flagStrata4}, \ref{prop:flagStrataMaps4}, Table \ref{table:Flag123-4}, and \cite[Proposition~A.6]{CoreyGrassmannians}.    
		
	Now suppose $\Gamma(\vecomega)$ is one of the remaining graphs. Let $\vec\sQ,\vec\sQ_a,\vec\sQ_b$ be the flag matroids from Example \ref{ex:forClosedImmersion}. Let $T$ be the subgraph obtained by removing the vertices of $\Gamma(\vecomega)$ whose flag matroid is isomorphic $\sQ$, together with their adjacent edges; in each case $T$ is a tree and the limit $\varprojlim_{T} \Fl$ is smooth and irreducible of dimension $6$ (following the same procedure as in the previous paragraph). 
		
		For simplicity, suppose that $T$ is obtained by removing a single vertex and its two adjacent edges, as is in the graphs 5 and 10. Then 
	    \begin{equation*}
		\xymatrix
		{{ \Fl(\vecomega)  }\ar[r]^-{} \ar[d]_-{} & \Fl(\vec{\sQ}) \ar[d]^{\varphi_{\vec\sQ,\vec\sQ_a}, \varphi_{\vec\sQ,\vec\sQ_b}}\\
			{\varprojlim_{T} \Fl } \ar[r]^{} & \Fl(\vec\sQ_a) \times \Fl(\vec\sQ_b)}
		\end{equation*} 	
		is a pullback diagram (we are implicitly applying Proposition \ref{prop:fullComplex2Graph}), and $\varphi_{\vec\sQ,\vec\sQ_a}, \varphi_{\vec\sQ,\vec\sQ_b}$ is a closed immersion by Example \ref{ex:forClosedImmersion}. The dimension of $\Fl(\vecomega)$ is at least $6$ by Theorem \ref{thm:closedimmersion} and the fact that $\init_{\vec{\sw}}\Fl^{\circ}(4)$, being flat degenerations of $\Fl^{\circ}(4)$, has dimension $6$. Therefore $\varprojlim_{\Gamma(\vecomega)} \Fl$ is smooth and irreducible of dimension $6$ by \cite[Lemma~6.11]{CoreySpinor}. 
		The remaining cases may be handled in a similar fashion, that is, by ``reattaching'' each of the vertices outside of $T$ one at a time. 
		The last statement follows from Theorem \ref{thm:closedimmersion} and \cite[Proposition~A.8]{CoreyGrassmannians}.
	\end{proof}

	\section{The Chow quotient $\chow{\Fl(4)}{H}$}
	\label{sec:Chow}
	
	In this final section, we study the normalization of the Chow quotient $\chow{\Fl(n)}{H}$ where $H \subset \PGL(n)$ is the diagonal torus. This is a compactification of $\Fl^{\circ}(n)/H$, which has a modular description inspired by \cite{Hu}. When $n=4$, $\chow{\Fl(n)}{H}$ is a sch\"on compactification of $\Fl^{\circ}(n)/H$, and we may use tropical techniques by \cite{HackingKeelTevelev2009} to determine its log canonical model. 
	
	\subsection{A torus action on $\Fl(\vec{r},n)$}
	
	Recall that the diagonal torus of $\GL(n)$ acts on $\C^{n}$ by scaling coordinates, and this induces an action of the diagonal torus  $H\subset \PGL(n)$ on the linear subspaces of $\C^{n}$ which preserves inclusions. In particular, this produces the action $H\curvearrowright \Fl(\vec{r},n)$ described in the introduction. 
	
	Let $\vec{\sQ}$ be the uniform $(\vec{r},n)$--flag matroid (see Remark \ref{rmk:matroidPolytope}), and  $L \subset N(\sQ_1)\times \cdots \times N(\sQ_s)$ the saturation of the image of the map in Formula \eqref{eq:definingL}.  The subtorus $T_L \subset T(\sQ_1) \times \cdots \times T(\sQ_s)$ may be identified with $H$, and the scaling action of $T_L \curvearrowright \P(\wedge^{r_1}\C^n) \times \cdots \times \P(\wedge^{r_s}\C^n)$ restricts to the above action $H \curvearrowright \Fl(\vec{r},n)$. The embedding $\Fl(\vec{r},n)\hookrightarrow \P(\wedge^{r_1}\C^n) \times \cdots \times \P(\wedge^{r_s}\C^n)$ induces an embedding of normalized Chow quotients
	\begin{equation*}
	\chow{\Fl(\vec{r},n)}{H} \hookrightarrow \chow{\left(\P(\wedge^{r_1}\C^n) \times \cdots \times \P(\wedge^{r_s}\C^n)\right)}{H}.
	\end{equation*}
	By the main results of \cite{KapranovSturmfelsZelevinsky}, the normalization of the Chow quotient of a projective toric variety by a subtorus of its dense torus is also a projective toric variety, and its polytope is a fiber polytope.  In our case, the polytope of $\chow{\left(\P(\wedge^{r_1}\C^n) \times \cdots \times \P(\wedge^{r_s}\C^n)\right)}{H}$ is the fiber polytope associated to the projection in Formula \eqref{eq:fiberPolytopeMap}, whose normal fan is $\pF(\vec{\sQ},\vec{\sa}) / L_{\R}$.

	\subsection{Sch\"on compactifications}
	Given a rational polyhedral \textit{pointed} fan $\pF$ in $N(\sQ)$, denote by $X(\pF)$ its toric variety. 
	Suppose $Y^{\circ} \subset T(\sQ)$ is a closed subvariety. The closure $Y$ of $Y^{\circ}$ in a $T(\sQ)$--toric variety $X(\pF)$ is a \textit{sch\"on compactification} if the multiplication map $T(\sQ)\times Y \to X(\pF)$ is smooth and surjective; in this case the \textit{support} of $\pF$, denoted $|\pF|$, is $\Trop(Y^{\circ})$. The variety $Y^{\circ}$ is \textit{sch\"on} if $Y^{\circ}$ has a sch\"on compactification, equivalently, $\init_{\sw} Y^{\circ}$ is smooth for all $\sw\in \Trop(Y^{\circ})$ \cite[Proposition~3.9]{HelmKatz}. If $Y^{\circ}$ is sch\"on and $\pF$ is any pointed rational polyhedral fan with $|\pF| = \Trop(Y^{\circ})$, then the closure of $Y^{\circ}$ in $X(\pF)$ is a sch\"on compactification \cite[Theorem~1.5]{LuxtonQu}.  
	
	We now specialize to the complete flag variety $\Fl(4)$. In this case, $\TFl^{\circ}(4) = \Dr(\vec{\sQ})$ where $\vec{\sQ}$ is the uniform $((1,2,3),4)$--flag matroid, see Remark \ref{rmk:matroidPolytope} and \S \ref{sec:inverseLimitsAndInitialDegenerations}. The fan $\pF_4 := \pF_{\Dr}(\vec{\sQ},\vec{\sa})/L_{\R}$ has support  $\TFlo{4} / L_{\R}$ and is described in Table \ref{tab:cones123-4}. Let
	\begin{gather*}
	    \sigma_1 = \R_{\geq 0}\langle \se_{12}, \se_{34}, \se_{1} + \se_{2} + \se_{12}\rangle + L_{\R} \hspace{15pt} \sigma_2 = \R_{\geq 0}\langle \se_{12}, \se_{34}, \se_{3} + \se_{4} + \se_{34}\rangle + L_{\R} \\ 
	    \tau = \R_{\geq 0}\langle \se_{12}, \se_{34}\rangle + L_{\R}
	\end{gather*}
	Then $\sigma_{1}$ and $\sigma_{2}$ glue along $\tau$ to form the convex polyhedral cone
	\begin{equation*}
	    \sigma = \R_{\geq 0}\langle \se_{12}, \se_{34}, \se_{1} + \se_{2} + \se_{12}, \se_{3} + \se_{4} + \se_{34}\rangle + L_{\R}
	\end{equation*}
	Doing this for the 3 pairs of cones in the 14th orbit in Table \ref{tab:cones123-4} determines a coarser fan $\pF_4'$ with support $\TFlo{4}/L_{\R}$. 
	
	Given a cone $\tau$ in $N(\sQ)_{\R}$, denote by $\langle \tau \rangle$ the linear subspace spanned by $\tau$, and $\Star(\tau)$ the star of $\tau$, viewed as a fan in $(N/\langle \tau \rangle)_{\R}$. 
	\begin{lemma}
	\label{lem:notPreservedUnderTranslation}
		For each cone $\tau$ in $N(\sQ)_{\R}$ such that  $\tau / L_{\R} \in \pF_4'$, the set $|\Star(\tau)| \subset (N/\langle\tau \rangle)_{\R}$ is not preserved under translation by a rational subspace of $(N/\langle\tau \rangle)_{\R}$. 
	\end{lemma}
	
	\begin{proof}
    By \cite[Lemma~7.2]{CoreyGrassmannians}, we must show that
    \begin{equation}
    \label{eq:tauEqualsIntersectionSigmas}
        \langle \tau \rangle = \bigcap_{\sigma\supset \tau} \; \langle \sigma \rangle
    \end{equation}
    where the intersection is taken over all cones $\sigma$ containing $\tau$ as a face such that $\sigma/L_{\R}$ is a maximal cone of $\pF_4'$. Given such a $\sigma$, define a matrix $A_{\sigma}$ such that $\langle \sigma \rangle = \ker A_{\sigma}$. Let  $\widetilde{A}_{\tau}$ be the matrix whose block rows are the $A_{\sigma}$ such that $\sigma\supset \tau$ and $\sigma/L_{\R}\in \pF_4'$ is maximal. The intersection on the right in Formula \eqref{eq:tauEqualsIntersectionSigmas} is the kernel of $\widetilde{A}_{\tau}$. Thus, for each cone $\tau/L_{\R}$ of $\pF_{4}'$, we must show that the rank of $\widetilde{A}_{\tau}$ equals $14 - \dim \langle \tau \rangle$. In fact, we need only show this for the non-maximal cones. This is a direct verification; we use OSCAR. 
 	\end{proof}

\noindent We prove Theorem \ref{thm:chowIntro} in the following form.	
	
	\begin{theorem}
		Consider the complete flag variety $\Fl(4)$. 
		\begin{enumerate}
			\item The normalized Chow quotient $\chow{\Fl(4)}{H}$ is a smooth, sch\"on, and simple normal crossings compactification  $\Fl_{\bone}^{\circ}(3)$.
			\item The closure $\Fl_{\bone}^{\lc}(3)$ of $\Flo{4} / H$ in $X(\pF_4')$ is a sch\"on and log canonical compactification of $\Fl_{\bone}^{\circ}(3)$. 
			\item The refinement $\pF_4 \to \pF_4'$ induces a log crepant resolution of singularities $\chow{\Fl(4)}{H} \to \Fl_{\bone}^{\lc}(3)$. 
		\end{enumerate}
	\end{theorem}
	
	\begin{proof}
		Consider statement (1). The normalized Chow quotient $\chow{\Fl(4)}{H}$ is the closure of $\Fl_{\bone}^{\circ}(3) = \Flo{4}/H$ in $X(\pF_4)$. Using Table \ref{tab:cones123-4}, one readily verifies that $\pF_4$ is strictly simplicial. As $\pF_4$ is supported on $\Trop(\Flo{4} / H)$, we have that $\chow{\Fl(4)}{H}$ is a sch\"on compactification of $\Flo{4} / H$, and is smooth with a simple normal crossings boundary, see \cite[Proposition~7.1]{CoreySpinor}.  
		
		Next, consider (2). The initial degenerations of $\Fl^{\circ}(4)/H$ are smooth and irreducible by Theorem \ref{thm:initFl4} and the isomorphism
		\begin{equation*}
		    \init_{\vec{\sw}} \Fl^{\circ}(n) \cong (\init_{\vec{\sw}'} \Fl^{\circ}(n)/H) \times H 
		\end{equation*}
		for any $\vec{\sw} \in \TFl^{\circ}(n)$; here $\vec{\sw}'$ denotes the image of $\vec{\sw}$ in  $(N(\sQ_1) \times \cdots \times N(\sQ_{n-1}))_{\R}/L_{\R}$. Compare to \cite[Lemma~7.3]{CoreySpinor} and \cite[Lemma~7.1]{CoreyGrassmannians}. The statement now follows from Lemma \ref{lem:notPreservedUnderTranslation} and \cite[Proposition~7.1]{CoreySpinor}. 
		Finally, (3) follows from (1) and \cite[Theorem~1.4]{Tevelev}.\qedhere
	\end{proof}

	\bibliographystyle{abbrv}
	\bibliography{bibliographie}
	\label{sec:biblio}

	\newpage
	
	\appendix
	
	\section{Summary of Notation}
	\label{app:notation}

{\footnotesize
\begin{table}[tbh!]
\begin{tabular}{rll}
   
   $[n]$ & $ = \{1,\ldots,n\}$ & p.\pageref{notn:setn} \\
   $\bone$ & the ``all ones vector'' & p.\pageref{notn:allone} \\
   $N$ &  $= \Z^n/\Zone$ & p.\pageref{notn:N} \\
   $M$ & $=\Hom(N,\Z)$ & p.\pageref{notn:M} \\
   $\bk{\su}{\sv}$ & The perfect pairing on $M\times N$ & p.\pageref{notn:uv}\\ 
   $\sQ$ & a finite set, frequently a matroid & p.\pageref{notn:Q} \\ 
   $(\sQ,\sa)$ & a point configuration & 
   p.\pageref{notn:pointConf}\\
   $\Delta(\sQ,\sa)$ & the polytope of $(\sQ,\sa)$ & 
   p.\pageref{notn:DeltaQ}\\
   $N(\sQ)$ &  $= \Z^{\sQ} / \Zone$ & p.\pageref{notn:NQ} \\
   $\omega$ & $ = (\sQ,\sa,\sw)$, a weighted point configuration & p.\pageref{notn:omega} \\
   $\sQ_{\sv}^{\sw}$ & a face in a coherent subdivision  & p.\pageref{notn:Qvw} \\
   $\sQ_{\sv}$ & $=\sQ_{\sv}^{0}$, a face of $(\sQ,\sa)$  & p.\pageref{notn:Qv} \\
   $\pQ(\omega)$ & the coherent subdivision associated to $\omega$ & p.\pageref{notn:subdOmega} \\
   $\vec{\sQ}$ & $ = (\sQ_1,\ldots,\sQ_s)$, frequently a flag matroid & p.\pageref{notn:vecQ} \\
   $\MS(\vec{\sQ},\vec{\sa})$ & The Minkowski sum of $(\vec{\sQ},\vec{\sa})$ & p.\pageref{notn:MinkowskiSum} \\
   $\vec{\omega}$ & $ = (\omega_1,\ldots,\omega_s) = (\vec{\sQ},\vec{\sa},\vec{\sw})$ & p.\pageref{notn:vecomega} \\
   $\pQ(\vecomega)$ & the coherent mixed subdivision associated to $\vecomega$ & p.\pageref{notn:subdVecOmega} \\
   $\pF(\vec{\sQ},\vec{\sa})$ & the fiber fan with its lineality space & p.\pageref{notn:fiberFan} \\
   $\se_{\lambda}$ & the image of the $\lambda$-th standard basis vector under $\Z^{\sQ}\to N(\sQ)$ & p.\pageref{notn:se} \\
   $M(\sQ)$ &  $= \Hom(N(\sQ),\Z)$ & p.\pageref{notn:MQ} \\
   $\se_{\lambda}^*$ & the dual of $\se_{\lambda}$ in $M(\sQ)$ & p.\pageref{notn:sedual}\\
   $\binom{[n]}{r}$ & the $r$-element subsets of $[n]$  & p.\pageref{notn:binomnr} \\
   $\pB(\vec{\sQ})$ & $=\bigcup_{k=1}^{s} \sQ_k$  & p.\pageref{notn:basesFlagMatroid} \\
   $\epsilon_1,\ldots,\epsilon_n$ & the image under $\Z^{n} \to N$ of the standard basis vectors & p.\pageref{notn:epsilon} \\
   $\epsilon_i^{*}$ & the dual of the $\epsilon_i$ in $M$ & p.\pageref{notn:epsilondual} \\
   $\epsilon_{\lambda}$ & $ = \epsilon_{i_1} + \cdots + \epsilon_{i_r}$ where $\lambda = \{i_1,\ldots,i_r\}$ & p.\pageref{notn:epsilonlambda} \\
   $\Dr(\vec{\sQ})$ & the flag Dressian & p.\pageref{notn:flagDr} \\
   $\Pi_n$ & the $(n-1)$--dimensional permutahedron & p.\pageref{notn:permutahedron} \\
   $T(\sQ)$ & the torus with cocharacter lattice $N(\sQ)$ & p.\pageref{notn:TQ} \\
   $\Fl(\vec{\sQ})$ & flag matroid stratum & p.\pageref{notn:FlvecQ} \\
   $\vec{\sP} \leq \vec{\sQ}$ & partial order induced by inclusion of faces & p.\pageref{notn:faceOrder} \\
   $\Fl(\vecomega)$ & an inverse limit of flag matroid strata & p.\pageref{notn:Flomega} \\
   $\Fl(n)$ & $=\Fl((1,2,\ldots,n-1),n)$, the complete flag variety & p.\pageref{notn:completeFlag} \\
   $\Sn{n}$ & the symmetric group on $[n]$ & p.\pageref{notn:Sn} \\
  \end{tabular}
\end{table}
}

	\section{Complete flag matroids on $[4]$ and matroidal subdivisions of $\Pi_4$}	
	\label{app:data}
	
	In this Appendix, we record the data used in \S\S\ref{sec:completeFlag}-\ref{sec:inverseLimitsCompleteFlag}.  In Table \ref{table:Flag123-4} we record $(\Sn{2}\times \Sn{4})$---orbit representatives of the flag matroids $\vec{\sQ}$ such that $\dim \Delta(\vec{\sQ}) = 3$, the maximum possible, along with its coordinate ring and internal flag matroids $\sQ' \leq \sQ$. Observe that for each representative $\sQ$, the set $\pB(\vec{\sQ})$ has the standard flag from Formula \ref{eq:standardFlag}, thus the coordinate ring $R_{\vec{\sQ}}^{x}$ is computed exactly as in \S \ref{sec:affineCoordinates}.  See Example \ref{ex:forClosedImmersion} for a sample computation. 
	
	Tables \ref{tab:Subdivisions-coarsest}, \ref{tab:Subdivisions-middle}, and \ref{tab:Subdivisions-finest} record all adjacency graphs of flag subdivisions  of the $n=4$ permutahedron, up to $(\Sn{2} \times \Sn{4})$---symmetry. The ordering is consistent with Table \ref{tab:cones123-4}. The flag matroids appearing are listed in the right-most column, and the number in parenthesis indicates which $(\Sn{2}\times \Sn{4})$--orbit from Table \ref{table:Flag123-4} the flag matroid belongs to.
	
		\newpage
	
	\begin{landscape}
	{\footnotesize
	\begin{table}[tbh!]
	\centering
	\begin{tabular}{|c|c|c|c|c|l|c|c|}
	\hline
	& \multirow{2}{3.1cm}{\centering \textbf{Flag matroids}  $\vec{\sQ}$ \textbf{(nonbases)}}  & \multirow{2}{65pt}{\centering $B_{\vec{\sQ}}^{x} / I_{\vec{\sQ}}^{x}$}  & \multirow{2}{153pt}{\centering $S_{\vec{\sQ}}^{x}$ \textbf{generators}} & \multirow{2}{22pt}{\centering $\dim$  $\Fl(\vec{\sQ})$} & \multirow{2}{4.1cm}{\centering \textbf{Flag matroids} $\vec{\sP} \leq \vec{\sQ}$ \\ \textbf{internal facets (nonbases)} }   & \multirow{2}{2.5cm}{\centering $\lambda \subset [4]$: $\vec{\sP}\cong $ \\ $\vec{\sQ}|{\lambda} \times \vec{\sQ}/\lambda$} 
	& \multirow{2}{1cm}{\centering $\dim$ \\ $\Fl(\vec{\sP})$} \\
	&  & & & &  &  & \\
	\hline
	\hline
	\textbf{1} & \multirow{4}{3.1cm}{\centering $\{2,4,13,24\}$}  & \multirow{4}{65pt}{\centering $\C[v,y,z]$}  & \multirow{4}{153pt}{\centering $v,y,z$} & \multirow{4}{22pt}{\centering $3$} & (a): $\{2,4,13,14,24,34,134\}$ & $2$ & $2$ \\
	& & & & &  (b): $\{2,4,12,13,23,24,123\}$ & $4$  & $2$ \\
	& & & & & (c): $\{2,4,13,24,123,134\}$ & $13$ &  $2$ \\
	& & & & & (d): $\{2,4,13,24,124,234\}$ & $24$ & $2$	\\
	\hline
	\textbf{2} &  \multirow{3}{3.1cm}{\centering $\{3,4,34,124\}$}  & \multirow{3}{65pt}{\centering $\C[u,x,y]$} & \multirow{3}{153pt}{\centering $u,x,y$} & \multirow{3}{22pt}{\centering $3$} & (a): $\{3,4,12,14,24,34,124\}$ & $3$ & $2$ \\
	& & & & & (b): $\{3,4,12,34,123,124\}$ & $34$ & $2$ \\
	& & & & & (c): $\{3,4,13,23,34,124\}$ & $124$ & $2$ \\
	\hline
	\textbf{3} &  \multirow{2}{3.1cm}{\centering $\{3,4,13,23,34\}$}   & \multirow{2}{65pt}{\centering $\C[u,y,z]$} & \multirow{2}{153pt}{\centering $u,y,z$} & \multirow{2}{22pt}{\centering $3$} & (a): $\{3,4,13,23,34,124\}$ & $3$ & $2$ \\
	& & & & & (b): $\{3,4,12,13,23,34,123\}$ & $4$  & $2$ \\
	\hline
	\textbf{4} & \multirow{2}{3.1cm}{\centering $\{3,13,14,23,34,134\}$}  & \multirow{2}{65pt}{\centering $\C[u,w,z]$} & \multirow{2}{153pt}{\centering $u,w,z$} & \multirow{2}{22pt}{\centering $3$} & (a):  $\{3,23,14,13,34,124,134\}$ & $3$ & $2$ \\
	& & & & & (b): $\{2,3,23,14,13,34,134\}$ & $134$ & $2$ \\
	\hline
	\textbf{5} & $\{2,4,24,124,234\}$ & $\C[v,x,y]$ & $v,x,y$ & $3$ & (a): $ \{2,4,13,24,124,234\}$ & $24$  & $2$ \\
	\hline
	\textbf{6} & \multirow{3}{3.1cm}{\centering $\{2,13,124\}$}   &  \multirow{3}{65pt}{\centering $\C[v,w,y]$} & \multirow{3}{153pt}{\centering $v,w,y$} & \multirow{3}{22pt}{\centering $3$} & (a): $\{2,13,14,34,124,134\}$ & $2$ & $2$ \\
	& & & & & (b):  $\{2,4,13,24,124,234\}$ & $13$ & $2$ \\
	& & & & & (c): $\{2,3,13,23,34,124\}$ & $124$ & $2$ \\
	\hline
	\textbf{7} & \multirow{3}{3.1cm}{\centering $\{2,3,23\}$} &  \multirow{3}{65pt}{\centering $\C[w,x,y,z]$} & \multirow{3}{153pt}{\centering $w,x,y,z,xz-y$} & \multirow{3}{22pt}{\centering $4$} & (a): $\{2,3,13,14,23,34,134\}$ & $2$ & $2$ \\
	& & & & & (b): $\{2,3,12,14,23,24,124\}$ & $3$ & $2$ \\
	& & & & & (c): $\{2,3,14,124,134\}$ & $23$ & $2$ \\
	\hline
	8 & $\{3,13,23,34\}$  &  $\C[u,w,y,z]$ & $u,w,y,z,uy-w$ & $4$ & (a): $\{3,13,23,34,124\}$ & $3$ & $2$ \\
	\hline
	9 & \multirow{2}{3.1cm}{\centering  $\{3,124\}$}  &  \multirow{2}{65pt}{\centering $\C[u,w,x,y]$} & \multirow{2}{153pt}{\centering $u,w,x,y, uy-w$} & \multirow{2}{22pt}{\centering $4$} & (a): $\{3,12,14,24,124\}$ & $3$ & $3$ \\
	& & & & & (b): $\{3,13,23,34,124\}$ & $124$ & $3$ \\
	\hline
	\textbf{10} & \multirow{2}{3.1cm}{\centering  $\{2,124\}$}  &  \multirow{2}{65pt}{\centering $\C[v,w,x,y]$} & \multirow{2}{153pt}{\centering $v,w,x,y,vy-wx$} & \multirow{2}{22pt}{\centering $4$} & (a): $\{2,13,14,34,124,134\}$ & $2$ & $2$ \\
	& & & & & (b): $\{2,3,13,23,34,124\}$ & $124$ & $2$ \\
	\hline
	\textbf{11} & \multirow{2}{3.1cm}{\centering $\{2,13\}$ } &  \multirow{2}{65pt}{\centering $\C[v,x,y,z]$} & \multirow{2}{153pt}{\centering $v,w,y,z,vz-w$} & \multirow{2}{22pt}{\centering $4$} & (a): $\{2,13,14,34,134\}$ & $2$ & $3$ \\
	& & & & & (b): $\{2,4,13,24,124,234\}$ & $13$ & $2$ \\
	\hline
	\textbf{12} & \multirow{2}{3.1cm}{\centering $\{13,24\}$ }  &  \multirow{2}{65pt}{\centering $\frac{\C[u,v,w,x,z]}{\langle uy-w \rangle}$} & \multirow{2}{153pt}{\centering  $u,v,w,x,z$} & \multirow{2}{22pt}{\centering $4$} & (a): $\{2,4,13,24,124,234\}$ & $13$ & $2$ \\
	& & & & & (b): $\{1,3,13,24,123,134\}$ & $24$ & $2$ \\
	\hline
	\textbf{13} &  \multirow{2}{3.1cm}{\centering $\{2\}$ }  & \multirow{2}{65pt}{\centering $\C[v,w,x,y,z]$} & \multirow{2}{153pt}{\centering $v$, $w$, $x$, $y$, $z$, $vy-wx$, $xz-y$, $vz-w$} & \multirow{2}{22pt}{\centering $5$} & \multirow{2}{125pt}{(a): $\{2,13,14,34,134\}$} & \multirow{2}{32pt}{\centering $2$} & \multirow{2}{23pt}{\centering $3$} \\
	& & & & & & & \\
	\hline
	\textbf{14} & \multirow{2}{3.1cm}{\centering $\{13\}$ } &  \multirow{2}{65pt}{\centering $\C[u,v,w,y,z]$} & \multirow{2}{153pt}{\centering $u$, $v$, $w$, $y$, $z$, $uy-w$, $w-vz-uy$ } & \multirow{2}{22pt}{\centering $5$} & \multirow{2}{125pt}{(a): $\{2,4,13,24,124,234\}$} & \multirow{2}{32pt}{\centering $13$} & \multirow{2}{23pt}{\centering  $2$} \\ 
	& & & & & & &  \\
	\hline
	\textbf{15} & \multirow{2}{3.1cm}{\centering $\emptyset$ } & \multirow{2}{65pt}{\centering $\C[u,v,w,x,y,z]$} & \multirow{2}{153pt}{\centering $u$, $v$, $x$, $w$, $y$, $z$, $ux-v$, $uy-w$, $vy-wx$, $xz-y$, $uxz-uy-vz+w$} & \multirow{2}{22pt}{\centering $6$} & & \multirow{2}{32pt}{\centering N/A} & \multirow{2}{23pt}{\centering N/A}  \\
	& & & & & & & \\
	\hline
    \end{tabular}
    \caption{Flag matroid strata of $\Fl(4)$}
    \label{table:Flag123-4}
	\end{table}
	}
	\end{landscape}

	\newpage

	\begin{table*}[h]
		\centering
		\begin{tabular}{ |c|l|l| }
			\hline
			& \textbf{Adjacency graph} & \textbf{Bases of the flag matroids (orbit)}\\
			\hline 
			\multirow{2}{3mm}{\centering $\mathbf{1}$} & \multirow{2}{5cm}{\centering \includegraphics[width=17mm]{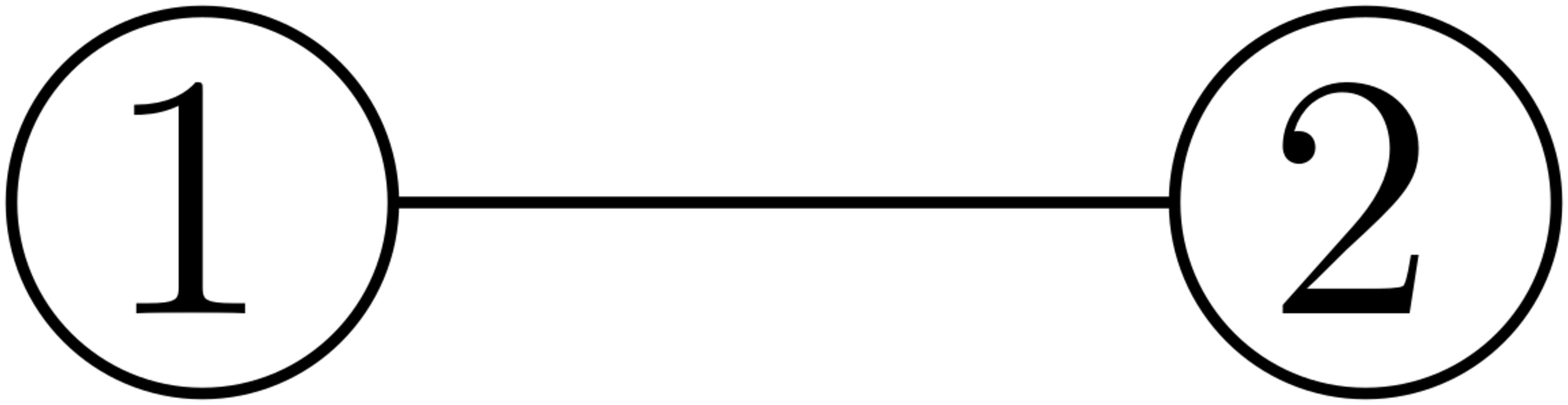}}& $\vec{\sQ}_1 :  \{1,2,3,4,12,13,14,123,124,134 \}$  (8) \\ 
			& & $\vec{\sQ}_2:\{2,3,4,12,13,14,23,24,34,123,124,134,234 \}$ (13)  \\ 
			\hline 
			\multirow{2}{3mm}{\centering $\mathbf{2}$} & \multirow{2}{5cm}{\centering \includegraphics[width=17mm]{G8.eps}}& $\vec{\sQ}_1 :  \{1,2,12,13,14,23,24,123,124\}$ (5) \\ 
			& &$\vec{\sQ}_2:\{1,2,3,4,13,14,23,24,34,123,124,134,234\}$ (14) \\ 
			\hline
			 \multirow{4}{3mm}{\centering $\mathbf{3}$}& \multirow{4}{5cm}{\centering \includegraphics[width=18mm]{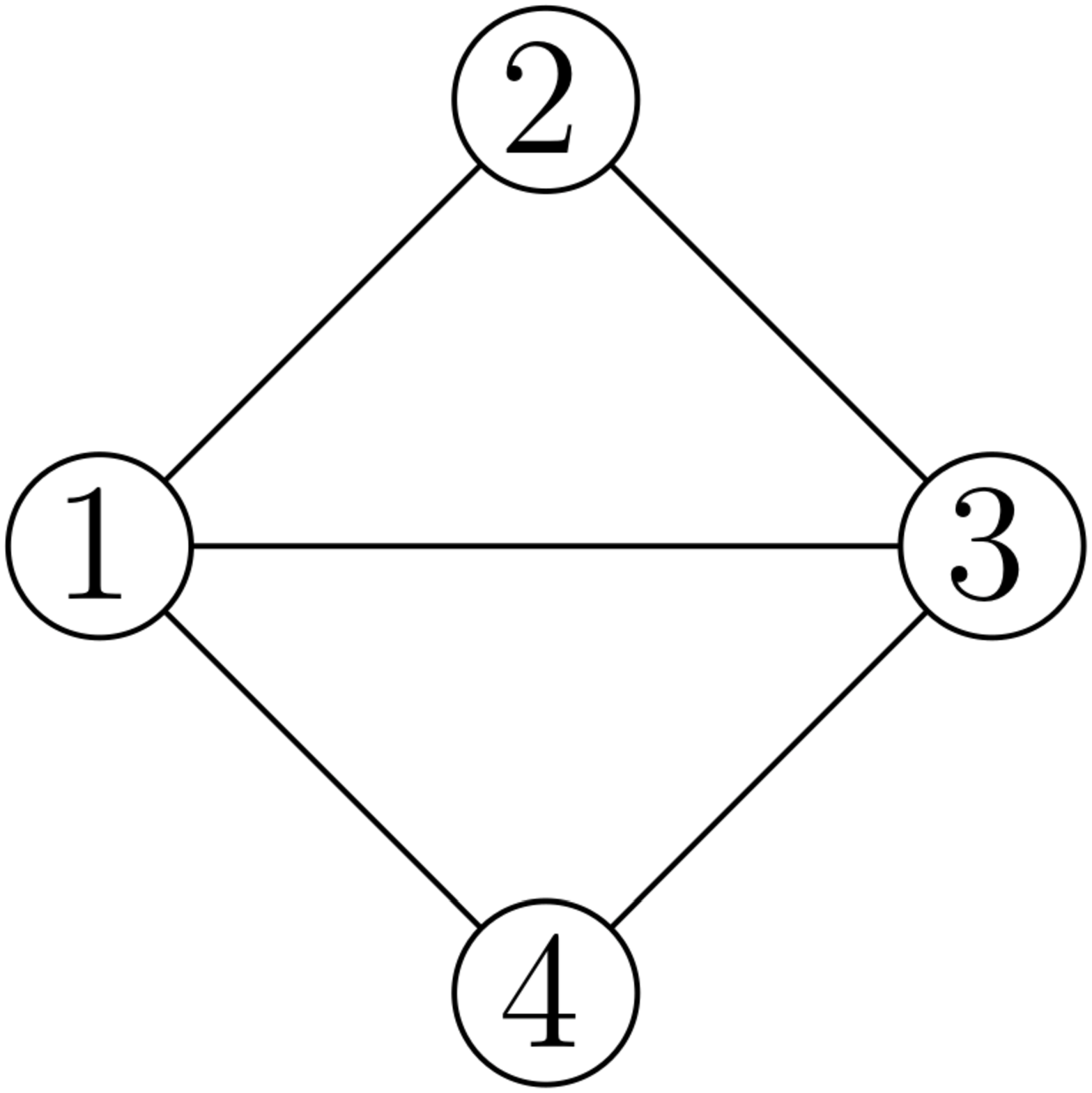}}& $\vec{\sQ}_1 :  \{1,2,3,4,12,13,14,23,24,123,124\}$ (7) \\ 
			& &$\vec{\sQ}_2:\{2,3,4,23,24,123,124,234\}$ (4) \\ 
			& &$\vec{\sQ}_3:\{3,4,13,14,23,24,34,123,124,134,234\}$ (7) \\ 
			& &$\vec{\sQ}_4:\{1,3,4,13,14,123,124,134\}$ (4) \\ 
			\hline
		\end{tabular}
		\caption{Coarsest matroidal subdivisions of $\Pi_4$}
		\label{tab:Subdivisions-coarsest}
	\end{table*}

	\begin{table*}[tbh]
		\centering
		\begin{tabular}{ |c|l|l| }
			\hline
			& \textbf{Adjacency graph} & \textbf{Bases of the flag matroids (orbit)}\\
			\hline 
			\multirow{3}{3mm}{\centering $\mathbf{4}$} & \multirow{3}{5cm}{\centering \includegraphics[width=26mm]{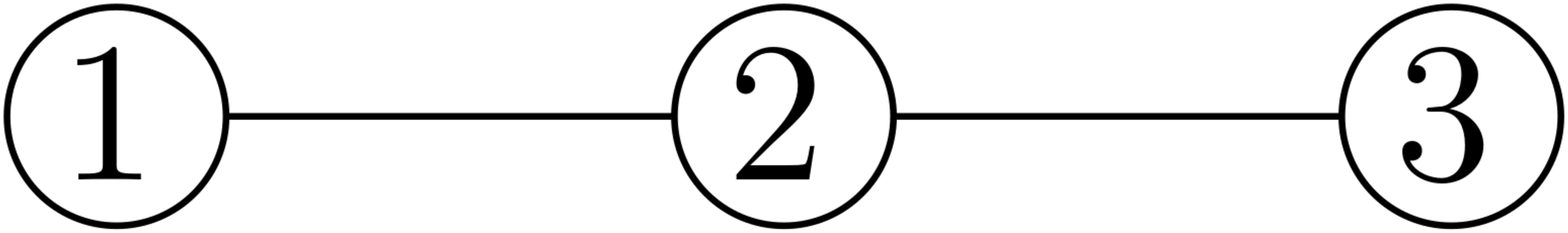}}& $\vec{\sQ}_1 :  \{1,2,3,4,12,13,14,123,124,134\}$ (8) \\ 
			& &$\vec{\sQ}_2:\{2,3,4,12,13,14,24,34,123,124,134,234\}$ (11) \\ 
			& &$\vec{\sQ}_3:\{2,3,12,13,23,24,34,123,234\}$ (5) \\ 			
			\hline
			\multirow{4}{3mm}{\centering $\mathbf{5}$} & \multirow{4}{5cm}{\centering \includegraphics[width=18mm]{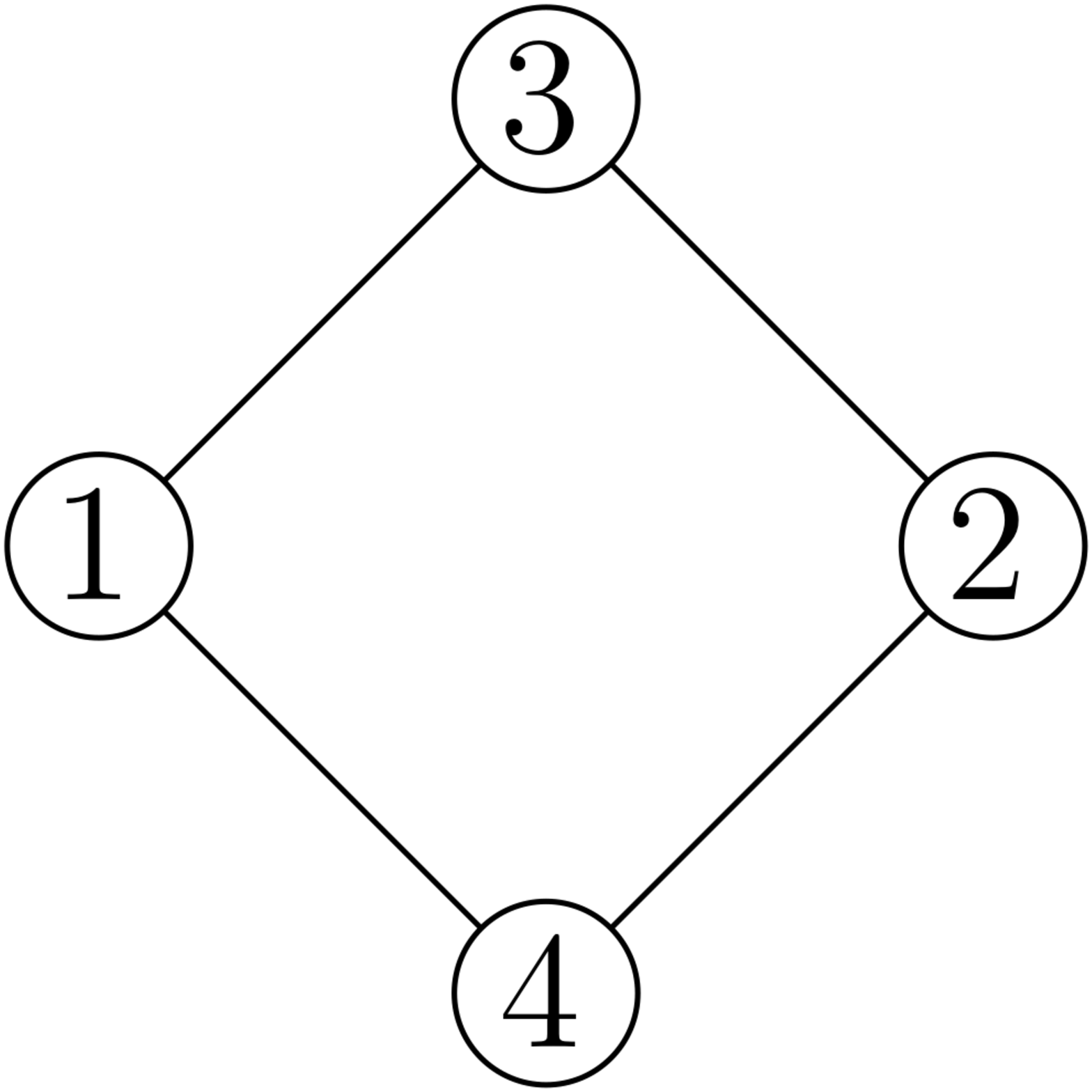}}& $\vec{\sQ}_1 :  \{1,2,3,12,13,123,124,134\}$ (4) \\ 
			& &$\vec{\sQ}_2:\{2,3,4,12,13,14,23,24,34,124,134,234\}$ (13) \\  
			& &$\vec{\sQ}_3:\{2,3,12,13,23,123,124,134,234\}$ (3) \\ 
			& &$\vec{\sQ}_4:\{1,2,3,4,12,13,14,124,134\}$ (3) \\ 
			\hline
			\multirow{3}{3mm}{\centering $\mathbf{6}$} & \multirow{3}{5cm}{\centering \includegraphics[width=26mm]{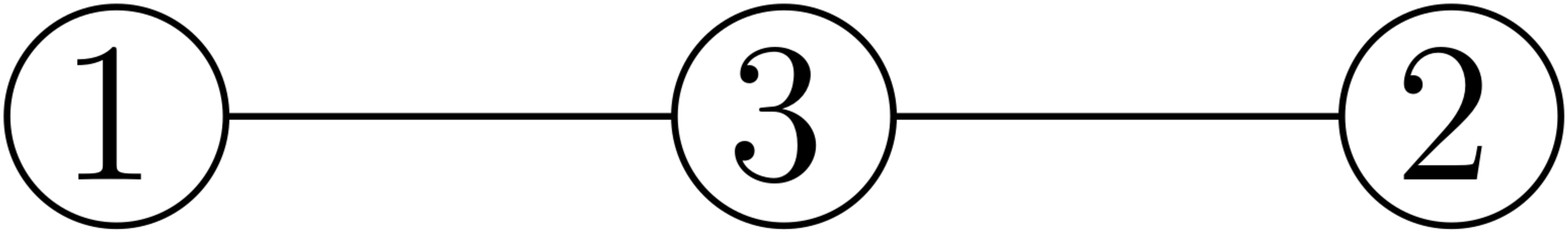}}& $\vec{\sQ}_1 :  \{1,2,3,4,12,13,14,123,124,134\}$ (8) \\ 
			 & &$\vec{\sQ}_2:\{2,3,4,23,24,34,123,124,134,234\}$ (8) \\ 
			& &$\vec{\sQ}_3:\{2,3,4,12,13,14,23,24,34,123,124,134\}$ (9) \\ 
			\hline
			\multirow{5}{3mm}{\centering $\mathbf{7}$} & \multirow{5}{5cm}{\centering \includegraphics[width=26mm]{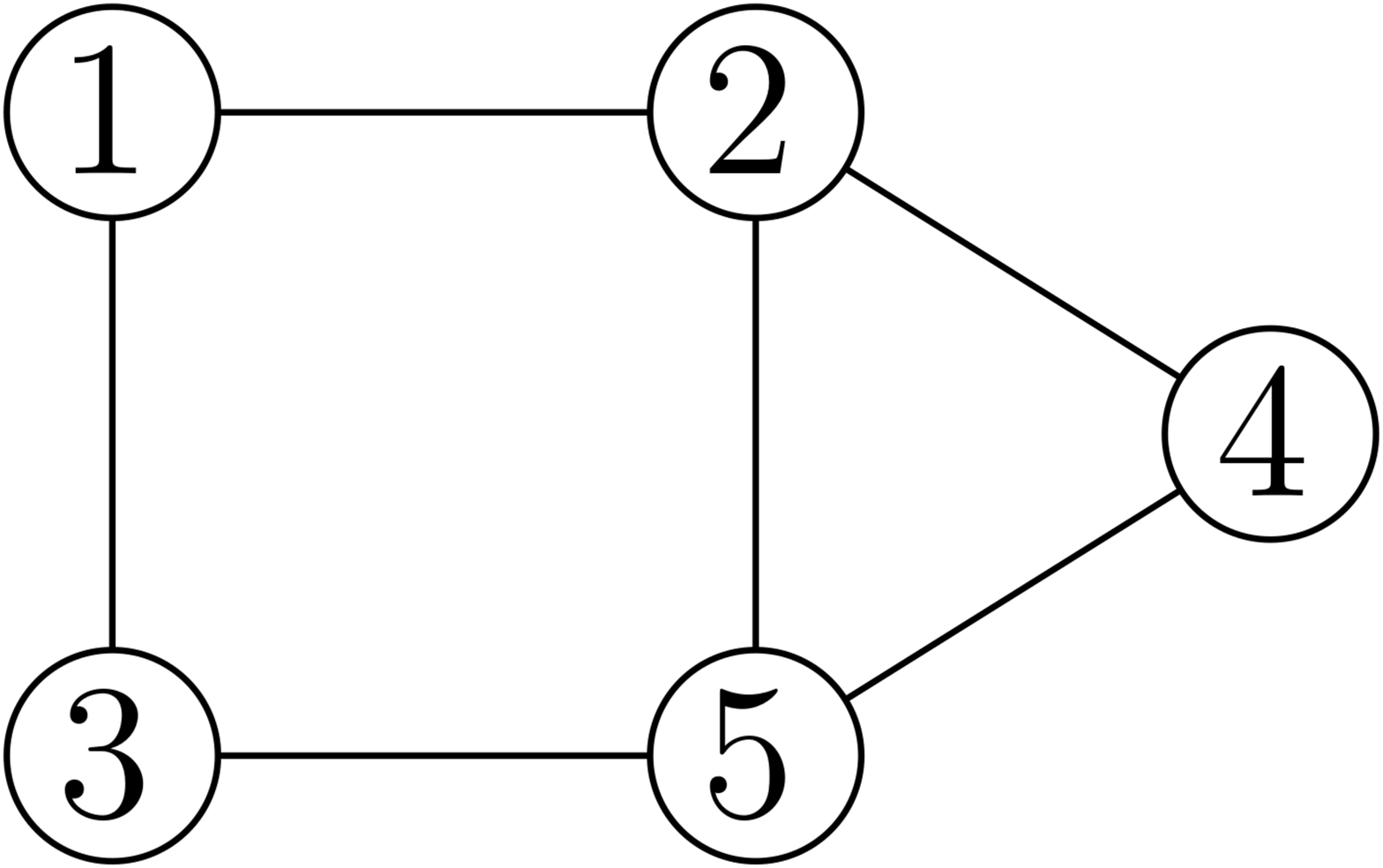}}& $\vec{\sQ}_1 :  \{1,2,3,4,12,13,14,123,124\}$ (3) \\ 
			& &$\vec{\sQ}_2:\{2,3,4,23,24,123,124,234\}$ (4) \\ 
			& &$\vec{\sQ}_3:\{3,4,13,14,23,24,34,123,124,134,234\}$ (7) \\ 
			& &$\vec{\sQ}_4:\{2,3,4,12,13,14,23,24,123,124\}$ (2) \\ 
			& &$\vec{\sQ}_5:\{1,3,4,13,14,123,124,134\}$ (4) \\ 
			\hline
			\multirow{3}{3mm}{\centering $\mathbf{8}$} & \multirow{3}{5cm}{\centering \includegraphics[width=26mm]{G12.eps}}& $\vec{\sQ}_1 : \{1,2,12,13,14,23,24,123,124\}$ (5) \\ 
			& &$\vec{\sQ}_2:\{3,4,13,14,23,24,34,134,234\}$ (5) \\ 
			& &$\vec{\sQ}_3:\{1,2,3,4,13,14,23,24,123,124,134,234\}$ (12) \\ 
			\hline
			\multirow{5}{3mm}{\centering $\mathbf{9}$} & \multirow{5}{5cm}{\centering \includegraphics[width=30mm]{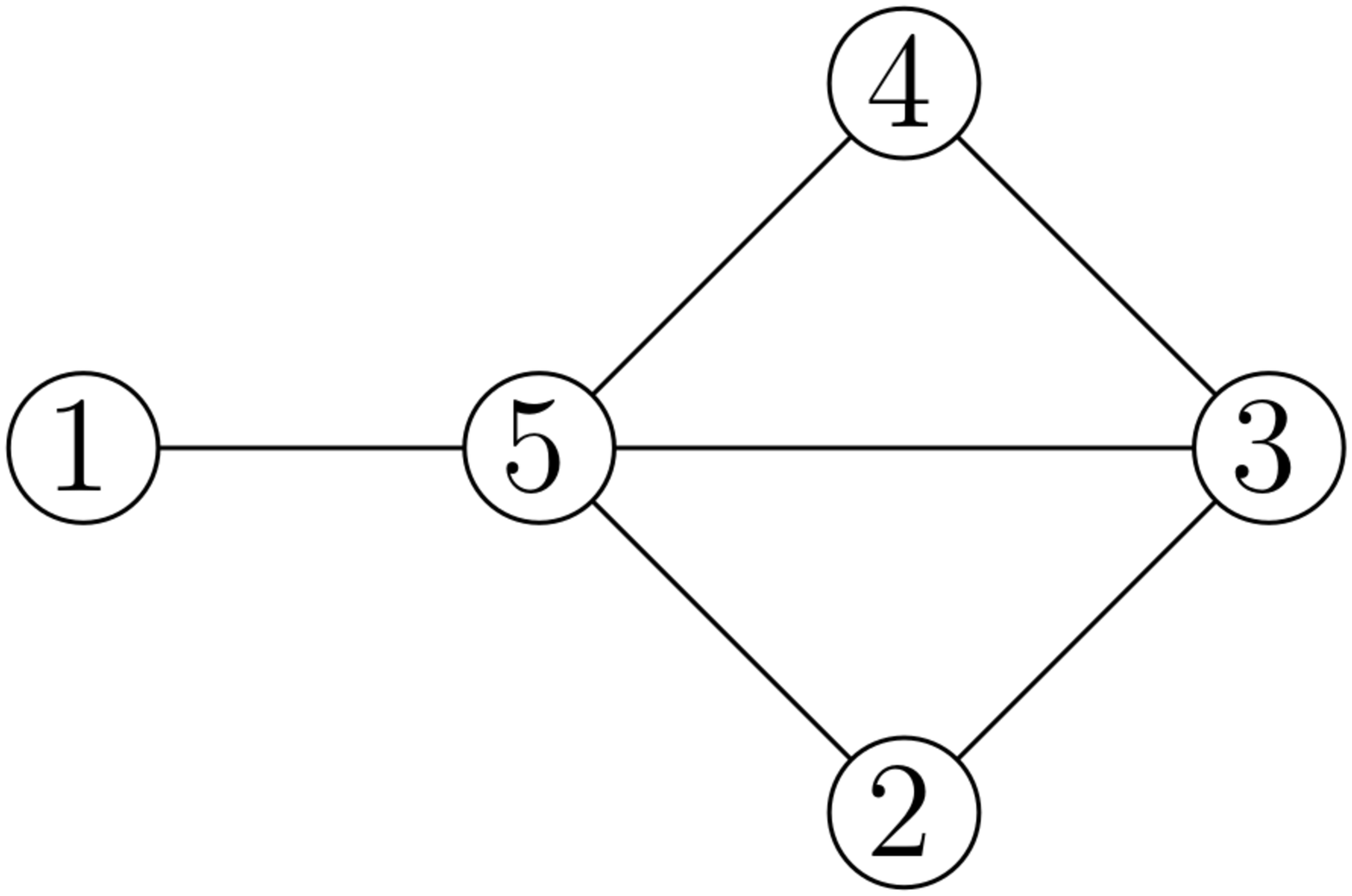}}& $\vec{\sQ}_1 : \{1,2,12,13,14,23,24,123,124\}$ (5) \\ 
			& &$\vec{\sQ}_2:\{1,2,3,13,23,123,134,234\}$ (4) \\ 
			& &$\vec{\sQ}_3:\{1,2,3,4,13,14,23,24,34,134,234\}$ (7) \\ 
			& &$\vec{\sQ}_4:\{1,2,4,14,24,124,134,234\}$ (4) \\ 
			& &$\vec{\sQ}_5:\{1,2,13,14,23,24,123,124,134,234\}$ (1) \\ 
			\hline
		\end{tabular}
		\caption{Intermediate matroidal subdivisions of $\Pi_4$}
		\label{tab:Subdivisions-middle}
	\end{table*}

	\begin{table*}[h]
		\centering
		\begin{tabular}{ |c|l|l| }
			\hline
			&\textbf{Adjacency graph} & \textbf{Bases of the flag matroids (orbit)}\\
			\hline
			\multirow{5}{4mm}{\centering $\mathbf{10}$} &\multirow{5}{5cm}{\centering \includegraphics[width=36mm]{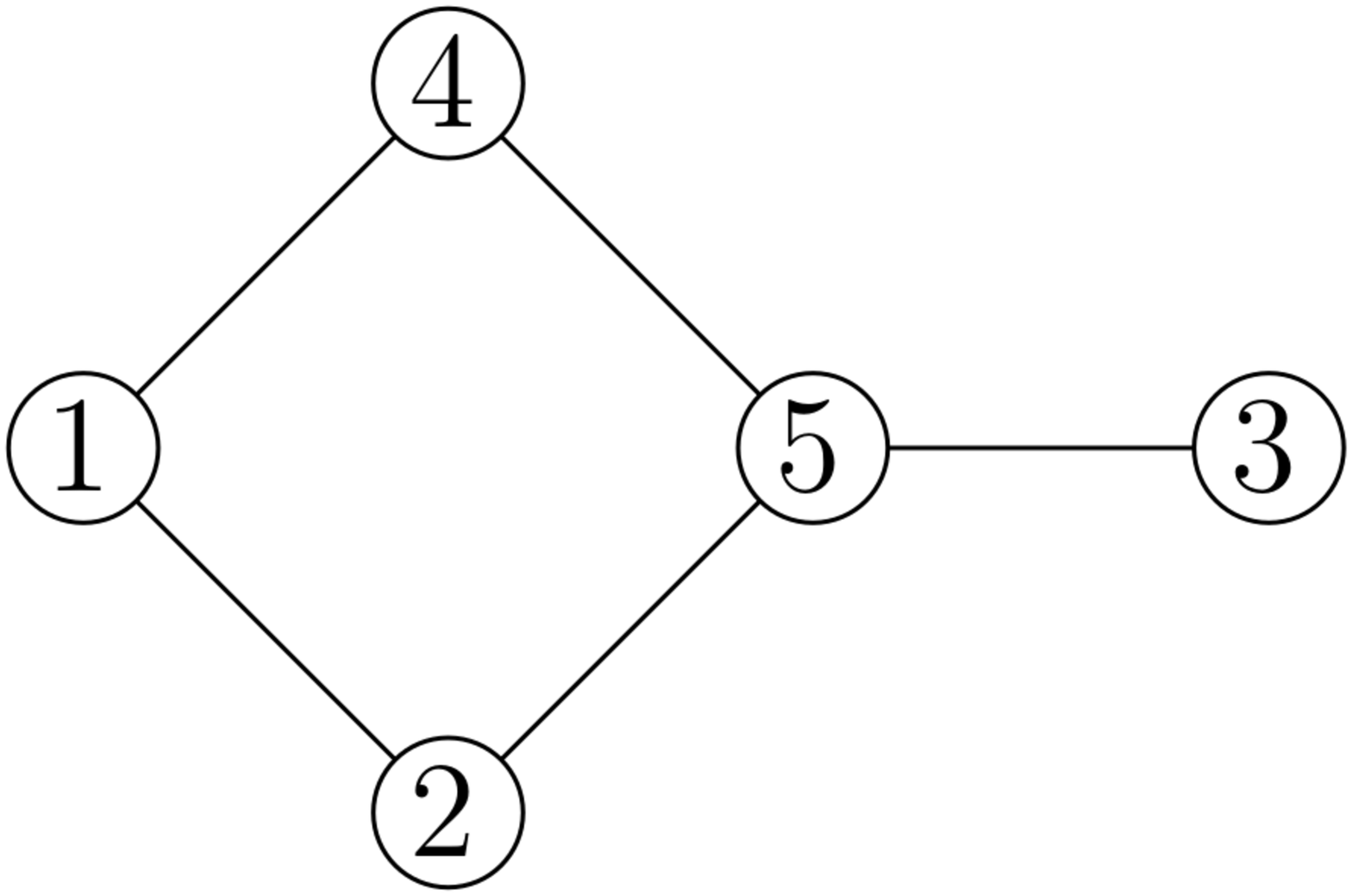}}& $\vec{\sQ}_1 : \{1,2,4,12,14,123,124,134\}$ (4) \\ 
			& &$\vec{\sQ}_2:\{1,2,3,4,12,13,14,123,134\}$ (3) \\ 
			& &$\vec{\sQ}_3:\{2,3,12,13,23,24,34,123,234\}$ (5) \\ 
			& &$\vec{\sQ}_4:\{2,4,12,14,24,123,124,134,234\}$ (3) \\
			& &$\vec{\sQ}_5:\{2,3,4,12,13,14,24,34,123,134,234\}$ (6) \\ 
			\hline
			\multirow{6}{4mm}{\centering $\mathbf{11}$} & \multirow{6}{5cm}{\centering \includegraphics[width=28mm]{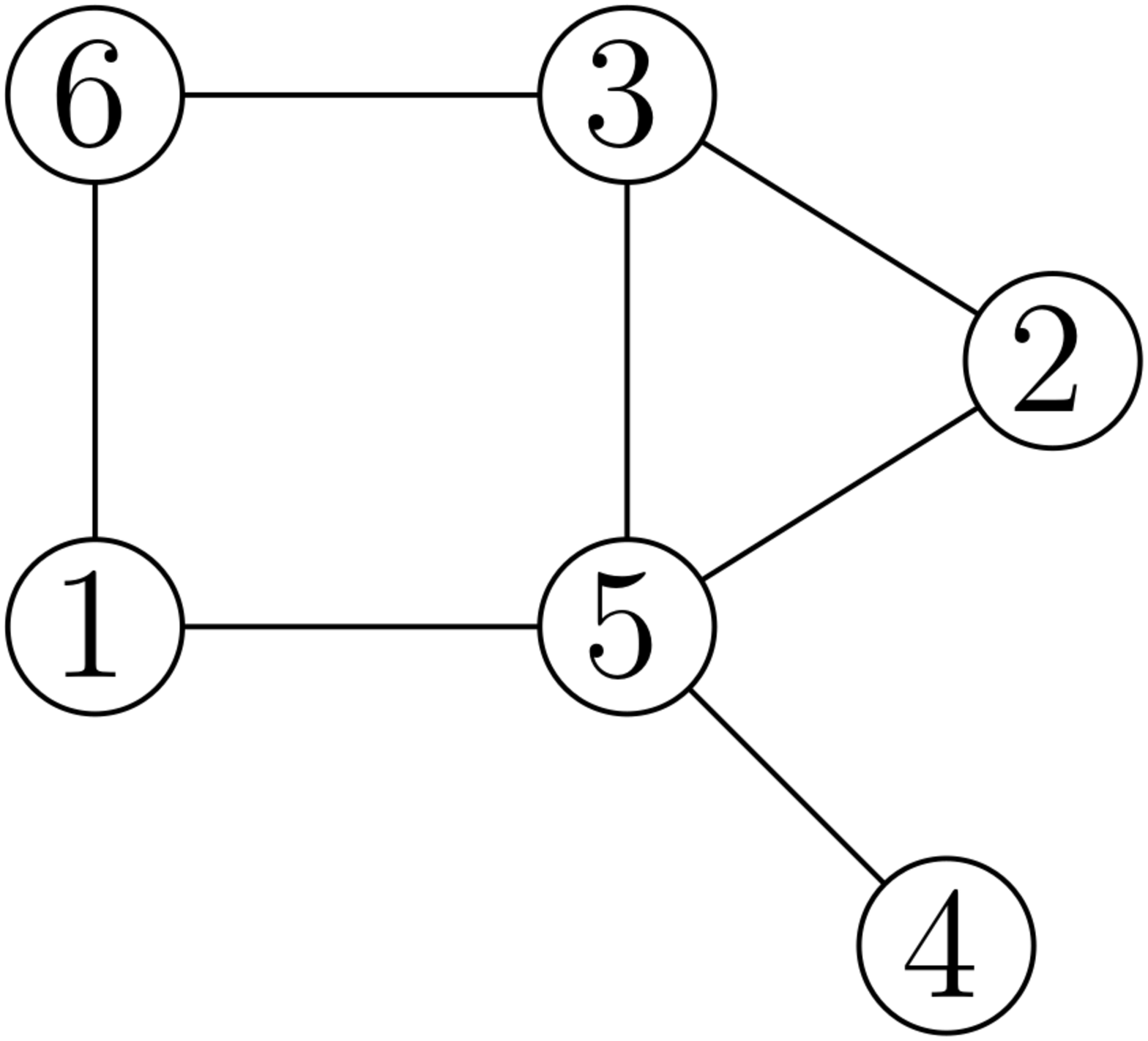}}& $\vec{\sQ}_1:\{1,2,3,12,13,123,124,134\}$ (4) \\
			& &$\vec{\sQ}_2:\{2,3,4,24,34,124,134,234\}$ (4) \\
			& &$\vec{\sQ}_3:\{2,3,4,12,13,14,24,34,124,134\}$ (2) \\
			& &$\vec{\sQ}_4:\{2,3,12,13,23,24,34,123,234\}$ (5) \\
			& &$\vec{\sQ}_5:\{2,3,12,13,24,34,123,124,134,234\}$ (1) \\
			& &$\vec{\sQ}_6:\{1,2,3,4,12,13,14,124,134\}$ (3) \\
			\hline
			\multirow{6}{4mm}{\centering $\mathbf{12}$} & \multirow{6}{5cm}{\centering \includegraphics[width=34mm]{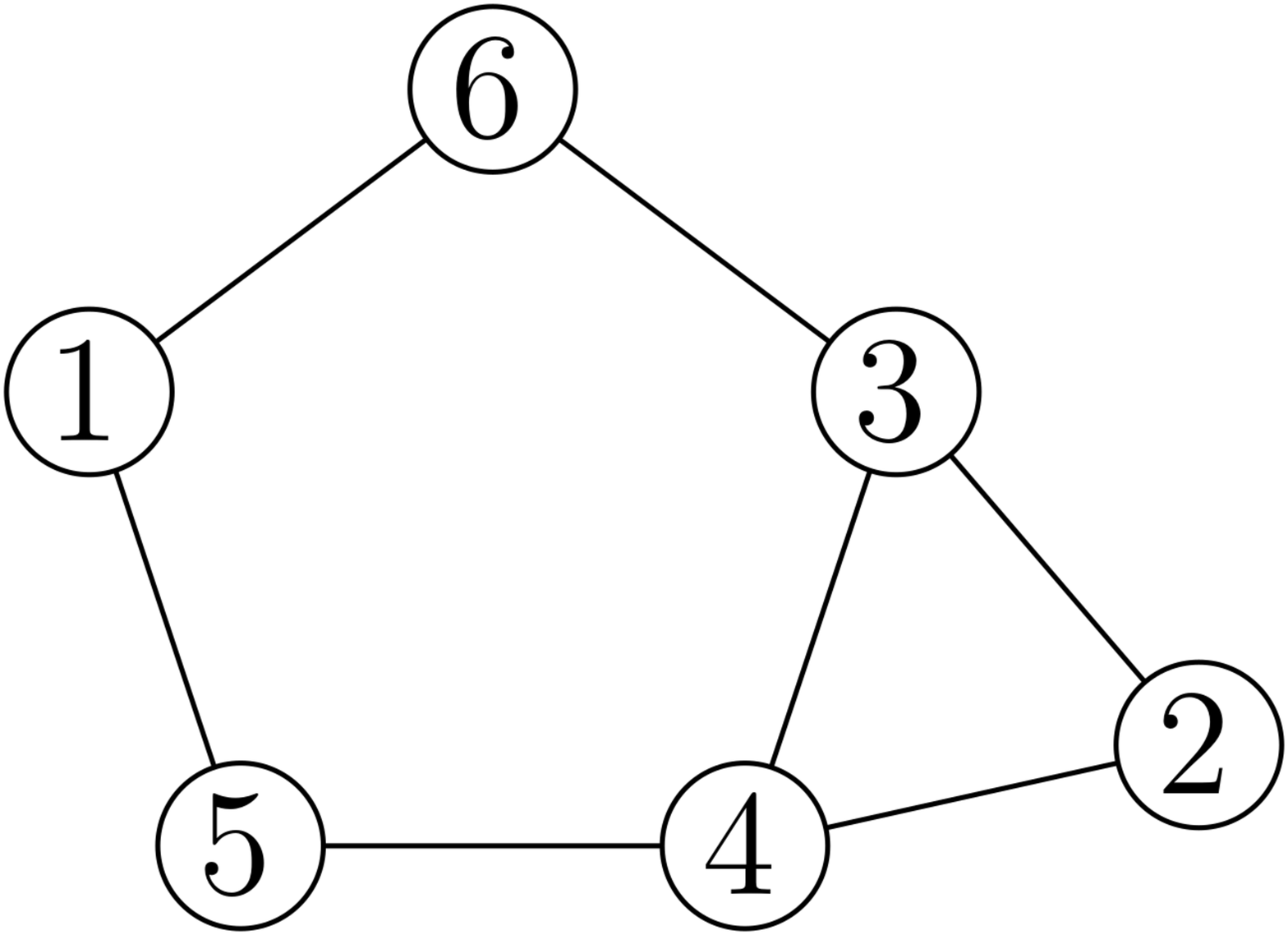}}& $\vec{\sQ}_1 : \{1,2,3,12,13,123,124,134\}$  (4) \\
			& &$\vec{\sQ}_2:\{2,3,4,24,34,124,134,234\}$ (4) \\
			& &$\vec{\sQ}_3:\{2,3,4,12,13,14,24,34,124,134\}$ (2) \\
			& &$\vec{\sQ}_4:\{2,3,12,13,23,24,34,124,134,234\}$ (2) \\
			& &$\vec{\sQ}_5:\{2,3,12,13,23,123,124,134,234\}$ (3) \\
			& &$\vec{\sQ}_6:\{1,2,3,4,12,13,14,124,134\}$ (3) \\
			\hline
			\multirow{6}{4mm}{\centering $\mathbf{13}$} & \multirow{6}{5cm}{\centering \includegraphics[width=32mm]{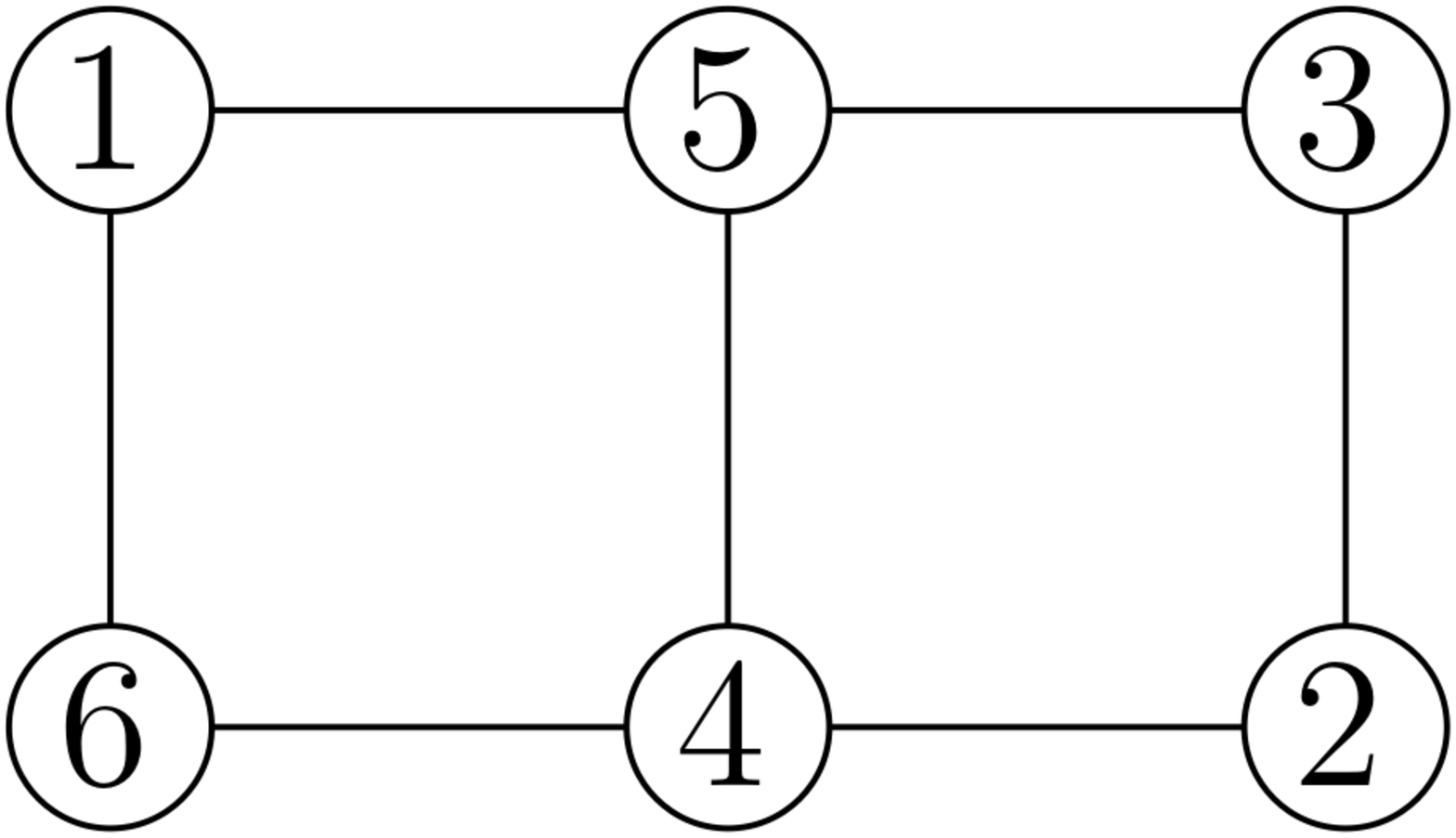}}& $\vec{\sQ}_1 : \{1,2,3,4,12,13,14,123,124\}$ (3) \\
			& &$\vec{\sQ}_2:\{3,4,23,24,34,123,124,134,234\}$ (3) \\
			& &$\vec{\sQ}_3:\{2,3,4,23,24,123,124,234\}$ (4) \\
			& &$\vec{\sQ}_4:\{3,4,13,14,23,24,34,123,124,134\}$ (2) \\
			& &$\vec{\sQ}_5:\{2,3,4,12,13,14,23,24,123,124\}$ (2)  \\
			& &$\vec{\sQ}_6:\{1,3,4,13,14,123,124,134\}$ (4) \\
			\hline
			\multirow{6}{4mm}{\centering $\mathbf{14}$} & \multirow{6}{5cm}{\centering \includegraphics[width=48mm]{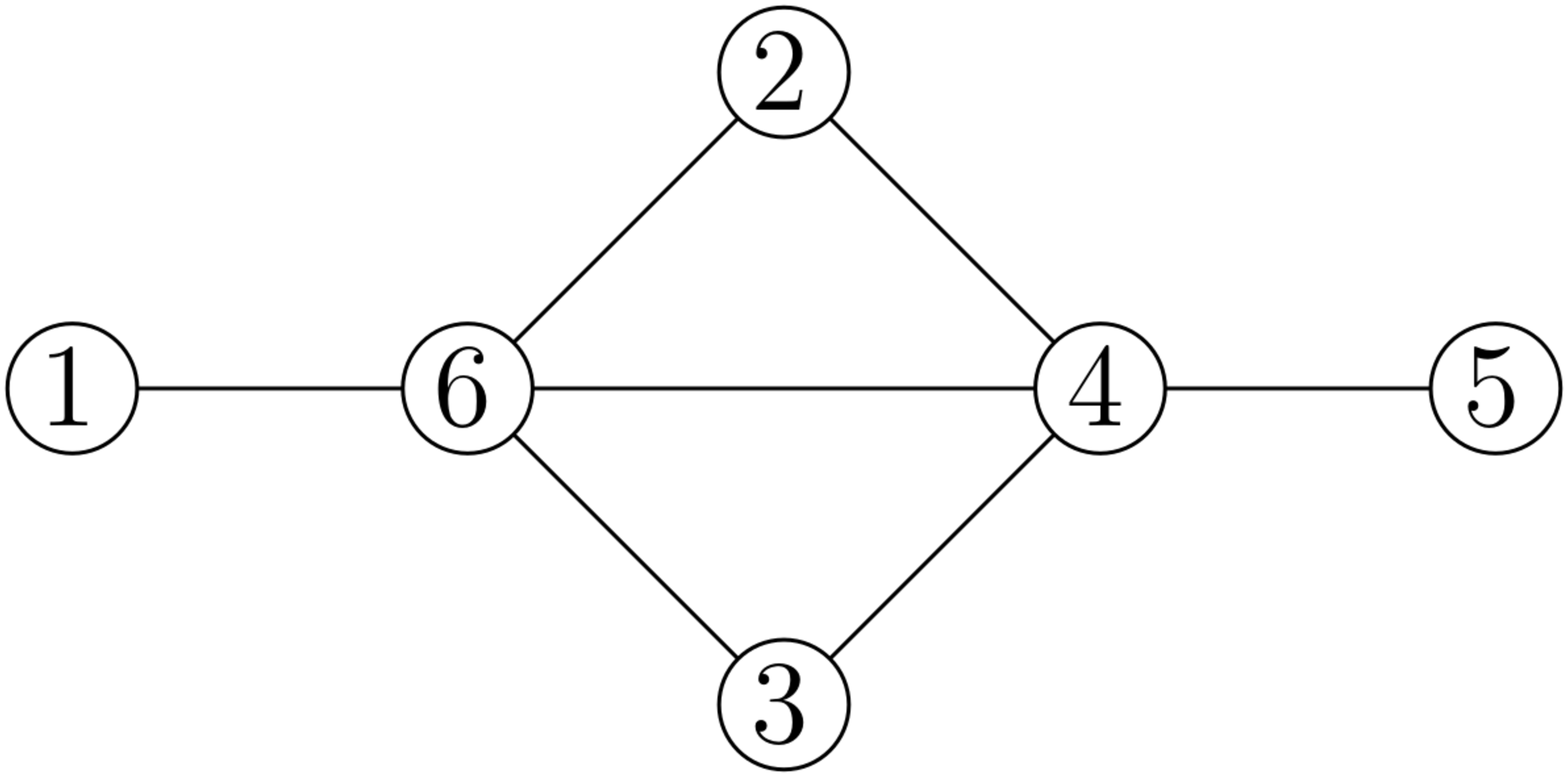}}&  $\vec{\sQ}_1 : \{1,2,12,13,14,23,24,123,124\}$ (5) \\
			& &$\vec{\sQ}_2:\{1,3,4,13,14,123,124,134\}$ (4) \\
			& &$\vec{\sQ}_3:\{2,3,4,23,24,123,124,234\}$ (4) \\
			& &$\vec{\sQ}_4:\{3,4,13,14,23,24,123,124,134,234\}$ (1) \\
			& &$\vec{\sQ}_5:\{3,4,13,14,23,24,34,134,234\}$ (5) \\
			& &$\vec{\sQ}_6:\{1,2,3,4,13,14,23,24,123,124\}$ (1) \\
			\hline
		\end{tabular}
		\caption{Finest matroidal subdivisions of $\Pi_4$}
		\label{tab:Subdivisions-finest}
	\end{table*}

\end{document}